\documentclass[11pt]{article}

\usepackage{bbding}
\usepackage{pifont}
\usepackage{amsfonts}
\usepackage{graphicx}
\usepackage{amsmath}
\usepackage{amsthm}
\usepackage{amssymb}
\usepackage{concrete}

\def\beq{\begin{equation}}
\def\eeq{\end{equation}}

\def\nd{\noindent}

\def\<{\leq}
\def\>{\geq}

\newtheorem{thm}{Theorem}[section]
\newtheorem{lem}{Lemma}[section]
\newtheorem{prop}{Proposition}[section]

\newtheorem{defi}{Definition}[section]
\newtheorem{rem}{Remark}[section]

\begin{document}
\title{Flat Connections and Brauer Type Algebras }
\author{Zhi Chen}
\date{}
\maketitle

\begin{abstract}
\noindent In this paper, we introduce a Brauer type algebra $B_G
(\Upsilon) $ associated with every pseudo reflection group and every
Coxeter group $G$. When $G$ is a Coxeter group of simply-laced type
we show $B_G (\Upsilon)$ is isomorphic to the generalized Brauer
algebra of simply-laced type introduced by Cohen, Gijsbers and Wales
({\it J. Algebra}, {\bf 280} (2005), 107-153). We also prove $B_G
(\Upsilon )$ has a cellular structure and be semisimple for generic
parameters when $G$ is a dihedral group or the type $H_3$ Coxeter
group. Moreover, in the process of construction,  we introduce a
further generalization of Lawrence-Krammer representation to complex
braid groups associated with all pseudo reflection groups.
\end{abstract}

\tableofcontents

\section{Introduction}

Brauer algebras $B_n (\tau )$ introduced by Brauer \cite{Br} are
certain algebras connected with representation theory and knot
theory. These algebras have natural deformations found by Birman ,
Wenzl \cite{BW} and by Murakami \cite{Mu}, which are called BMW
algebras.

Like many objects related to Lie theory,  Brauer algebras can be
generalized to other general reflection groups (here by a general
reflection group we mean a Coxeter group or a pseudo reflection
group). In \cite{Ha} H\"{a}ring-Oldenburg introduced the Cyclotomic
Brauer algebras and cyclotomic BMW algebras associated with the
$G(m,1,n)$ type pseudo reflection groups. Slightly later in
\cite{CGW1} Cohen, Gijsbers and Wales introduced  a
 Brauer type algebra and a BMW type algebra for each simply laced
 Coxeter group.  It is proved that these new Brauer type algebras
share many nice algebraic properties with Brauer algebras like
semisimplicity for generic parameters in Cohen-Frenk-Wales
\cite{CFW}, supporting Cellular structures Cohen-Frenk-Wales op. cit
in the sense of Graham-Lehrer \cite{GL}. In \cite{GH} Goodman and
Hauschild introduced Affine BMW algebras as a generalization of BMW
algebras to affine $\hat{A}_n$ type.

It is  asked in Cohen-Gijsbers-Wales \cite{CGW1}  whether there
exist Brauer type algebras and BMW type algebras for non simply
laced Coxeter groups. With the help of KZ connections, we introduce
in this paper a  Brauer type algebra $B_G (\Upsilon)$ for each
general reflection group $G$ . we also justify that the algebras is
a suitable candidates for general Brauer type algebra from the
following aspects.

\begin{itemize}
\item  If $W_{\Gamma}$ is a simply laced Coxeter group of type
$\Gamma$, the algebra $B_{W_{\Gamma}} (\Upsilon)$ coincides with the
simply laced Brauer algebra of type $\Gamma$ introduced in
Cohen-Gijsbers-Wales ibid  (Definition 8.2, Theorem 8.4). if $G$ is
a type $G(m,1,n)$ pseudo reflection group, the cyclotomic Brauer
algebra introduced by H{\" a}ring-Oldenberg in \cite{Ha} appears as
a direct component of our algebra $B_G (\Upsilon)$ (Theorem 8.6 ).

\item When $G$ is a finite pseudo reflection group (including all finite Coxeter groups),
  $B_G (\Upsilon)$ supports a nicely shaped flat connection on the
  complementary space of reflection hyperplanes of $G$. Existence of
  such a connection is a general phenomenon among finite
 pseudo reflection groups Brou\'e -Malle-Rouquier \cite{BMR}, and simply laced Brauer algebras (Theorem
 3.2, Theorem 5.3 ). These flat connections insure in some sense that $B_G
 (\Upsilon)$ can be deformed to certain BMW type algebras.

 \item Every $B_G (\Upsilon)$ induces a generalized Lawrence-Krammer
  representation  of the associated complex braid group $A_G$ (Theorem 5.2).

  \item When $G$ is finite, $B_G (\Upsilon)$ is a finite dimensional
  algebra containing $\mathbb{C} G$ (Theorem 5.1 ). There exists a natural
  anti-involution in $B_G (\Upsilon)$ (Lemma 5.5) which may be used to construct
  a cellular structure.

  \item When $G$ is a dihedral group or a $H_3$ type Coxeter group,
  the algebra $B_G (\Upsilon)$ has a cellular structure, and is
  semisimple for generic $\Upsilon$ (Sections 6-7).

\end{itemize}

Before giving the definition, we set up some notations which will be
used throughout this paper.  Let $G\subset U(V)$ be a finite pseudo
reflection group. Denote by $R$ the set of pseudo reflections in
$G$, and let $\mathcal {A} = \{H_i \} _{i\in P}$ be the set of
reflection hyperplanes. For $s\in R$, define $i(s)\in P$ by
requesting $H_{i(s)}$ to be the reflection hyperplane of $s$. We
also denote the reflection hyperplane of $s$ as $H_s$.  Intersection
of a subset of $\mathcal {A}$ is called an edge. The action of $G$
on $V$ induces an action of $G$ on $\mathcal {A}$ naturally. For
$i,j \in P$, let $R(i ,j ) =\{s\in R\ |\ s(H_j ) =H_i \ \}$. For
$i\in P$, let $G_i $ be the subgroup of $G$ consisting of elements
that fixing $H_i $ pointwise, let $m_i = |G_i |$, and let $s_i $ be
the unique element in $G_i $ with exceptional eigenvalue $e^{\frac
{2\pi \sqrt{-1}}{m_i }}$.

 Set $M_G = V- \cup _{i\in P}
 H_i$. For $i\in P$, choose a
 linear function $\alpha _i$ such that $H_i =\ker \alpha _i$ and define
 $\omega _i = \frac{d\alpha _i }{\alpha _i }$, which are holomorphic
 closed 1-forms on $M_G$. We write $s_1 \sim s_2 $ for $s_1 ,s_2 \in R$ if $s_1$ and $s_2$ are in the
same conjugacy class, and write $i\sim j $ for $i,j \in P$ if
$w(H_i) =H_j $ for some $w\in G$. Chose $0\neq \mu _s \in
\mathbb{C}$ for every $s\in R$ and $m_i \in \mathbb{C}$ for every
$i\in P$ such that $\mu _{s_1 } =\mu_{s_2 }$  if $ s_1 \sim s_2
 $ , $m_i =m_j $  if $i\sim j.$ The data $\{\mu _s ,m_i \}_{s\in \mathcal {R} ,i\in P }$ will be
denoted by one symbol $\Upsilon $. A well-known theorem by Steinberg
says that $G$ acts on $M_G$ freely. We denote the group   $\pi _1
(M_G /G )$, $\pi_1 (M_G)$ as $A_G$,$P_G$ respectively,  which are
called complex braid groups and complex pure braid groups by many
authors.

The original model for above setting came from the symmetric group
$S_n$.  First, $S_n$ is realized as a reflection group acting on
 $\mathbb{C}^n$ by permuting the basis elements, whose reflection hyperplanes are
 $\{ H_{i,j}={\rm Ker}\,(z_i -z_j )\}_{1\leq i<j\leq
 n}$. Then  $M_{S_n} =Y_n =  \{ (z_1 ,\cdots ,z_n )\in \mathbb{C}^n | z_i \neq z_j \ for\ any\ i\neq j
 \}$, which is the configuration space of $n$ different points on
 $\mathbb{C}$.  The differential form associated with $H_{i,j}$ is
 $\frac{dz_i -dz_j}{z_i -z_j}$. The associated group $A_G$ is just
 the $n$ string braid group $B_n$.

 For $i\neq j$ we denote $H_i \pitchfork H_j $ if
$\{ k \in P \ |\ H_i \cap H_j \subset H_k \} =\{ i,j \}$. A
codimension 2 edge $L$ will be called a crossing edge if there
exists $i,j\in P$ such that $H_i \pitchfork H_j$ and $L= H_i \cap
H_j$, otherwise $L$ will be called a noncrossing edge.


\begin{defi}
  The algebra $B_G (\Upsilon )$
associated with pseudo reflection group $(V,G)$ is generated by the
set
  $\{ T_{w} \}_{w\in G } \cup  \{ e_i \}_{i\in P}$ which satisfies the following
relations.

\begin{itemize}
 \item[$(0)$]  $T_{w_1} T_{w_2} = T_{w_3 }$ if $w_1 w_2 =w_3 $.
\item[$(1)$] $ T_{s_i }  e_i  = e_i T_{s_i } = e_i $, for $i\in P$.
\item[$(1)^{'}$] $T_w e_i =e_i T_w = e_i $, for $w\in G$ such that $w(H_i) = H_i$, and  $H_i \cap V_w $ is a noncrossing
edge. Where $V_w =\{ v\in V | w(v)=v  \}$.
\item[$(2)$] $e_i ^2 = m_i e_i $ .
\item[$(3)$] $T_w e_j  = e_i T_w $ , if $w \in G  $ satisfies $w(H_j
)=H_i$.
\item[$(4)$] $e_i e_j = e_j e_i $,
 if  $H_i \pitchfork H_j$ .
 \item[$(5)$] $ e_i e_j = (\sum _{s\in R(i,j) } \mu _s T_s )e_j = e_i (\sum _{s\in R(i,j) } \mu _s T_s )$ ,
 if $H_i \cap H_j$ is a noncrossing edge,
  and $R(i,j) \neq \emptyset $.
\item[$(6)$] $e_i e_j =0 $, if $H_i \cap H_j$ is a noncrossing edge, and $R(i,j) =\emptyset $.
\end{itemize}

\end{defi}
 Each one of these relations can be thought of as generalization
 of certain
relation in the Brauer algebra $B_n (\tau)$. Relation $(4)$ is the
generalization of $e_{i,j}e_{k,l}=e_{k,l}e_{i,j}$ for different
$i,j,k,l$,  where $e_{i,j}$ is explained in the following Figure 1.
Relation $(6)$ can be seen as a special case of relation $(5)$, they
are generalizations of the relation $e_{i,j} e_{j,k} = s_{i,k}
e_{j,k} =e_{i,j} s_{i,k}$ in $B_n (\tau)$. Where
 $s_{i,j}$ is the $(i,j)$ permutation.  Relation $(1)^{'}$
resembles relation $(o)$ in Definition 2.1 of cyclotomic Brauer
algebras.  Motivation of introducing these algebras is as follows.

\begin{figure}[htbp]

  \centering
  \includegraphics[height=3cm]{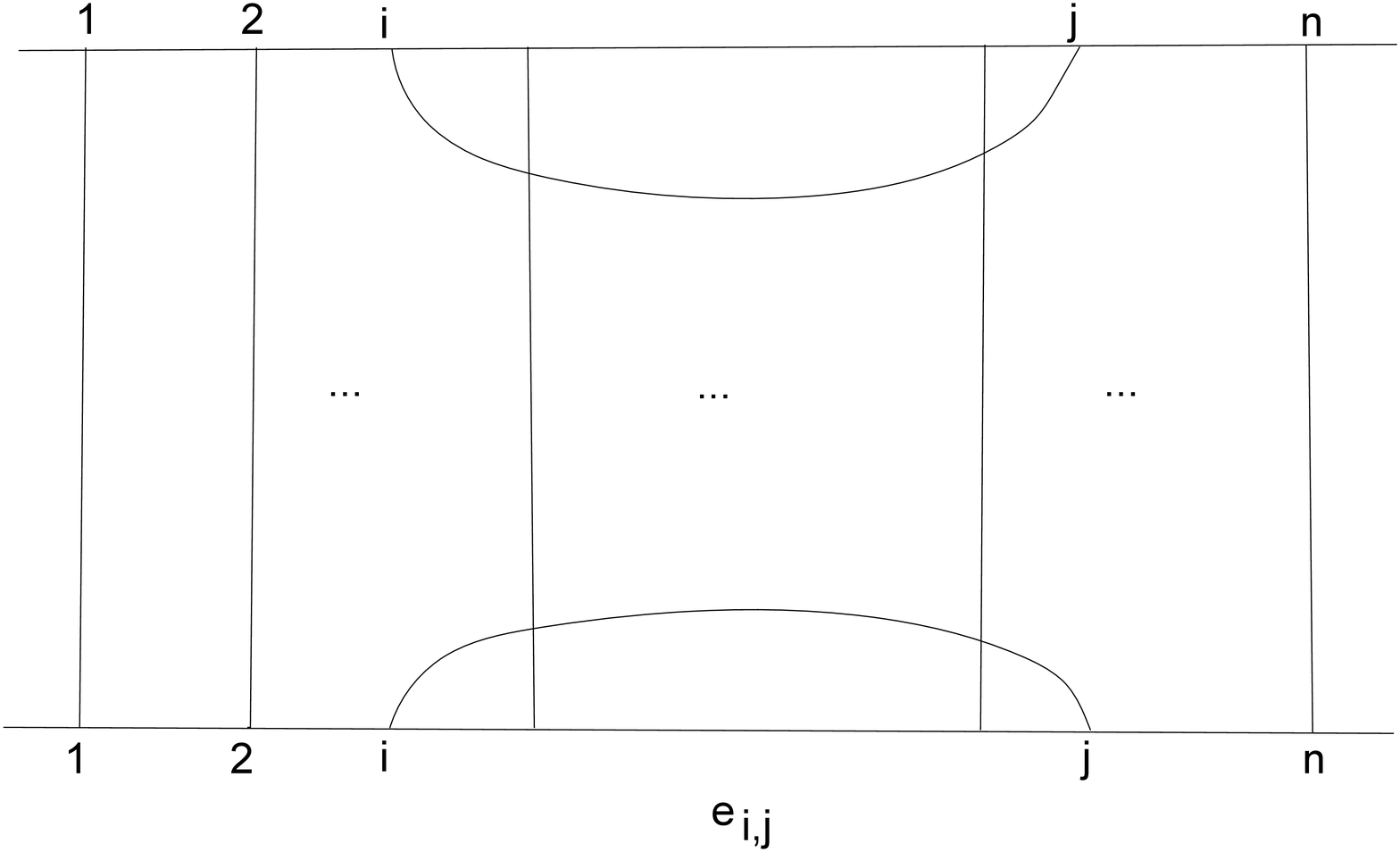}
  \caption{The element $e_{i,j}$}
\end{figure}

It is well-known that the group algebras of Coxeter groups and
pseudo reflection groups have deformations called Hecke algebras. It
is possibly less well-know that for any finite Coxeter group or any
pseudo reflection group, the infinitesimal deformation to the
corresponding Hecke algebra can be described by a KZ connection with
nice shape Cherednik \cite{Che} Brou\'e -Malle-Rouquier \cite{BMR}.

When $G$ is a finite Coxeter group, the KZ connection $\Omega_G$ (
$\Omega ^{'}_G$ ) describing deformation of $G$ to the $G$-type
Hecke algebra with equal parameter $H_G (q)$ ($G$-type Hecke algebra
with unequal parameters $H_G (\bar{q})$) is:
\begin{equation}
\Omega _G = \kappa \sum _{i\in P} s_i \omega _i,\ \ \ \ \ \ \ \ \
\Omega ^{'}_G =\sum _{i\in P} \kappa_i s_i \omega_i,
\end{equation}
where $\kappa$, $\kappa_i$'s are constants such that $\kappa_i
=\kappa_j$ if $i\sim j$. $\Omega _G$ is a formal connection on
$M_G$. Suppose $(U, \rho )$ is any representation of $G$. Then
$\Omega _G$ gives a $G$-invariant, flat connection $\rho (\Omega _G
) = \kappa \sum _{i\in P} \rho (s_i ) \omega_i  $
 on the bundle $M_G \times U$, which further induces a flat connection on the quotient bundle
 $M_G \times _G U$ whose monodromy representation factors
 through $H_G (q )$ for suitable $q$. As a special case, the
 KZ connection for symmetric group $S_n$ is
\begin{equation}
\Omega _n = \kappa \sum _{1\leq i<j\leq n} s_{i,j} \frac{dz_i
-dz_j}{z_i
 -z_j}.
\end{equation}

The KZ  connection for a pseudo reflection group $G$ has the form:

\begin{equation}
\Omega _G = \sum _{i\in P} \left(\sum _{s:\, i(s) =i } \mu_s s
\right) \omega _i ,
\end{equation}
where $\mu_s$ are constants such that $\mu_s =\mu_{s^{'}}$ if $s\sim
s^{'}$. It plays an important role in Brou\'e -Malle-Rouquier
\cite{BMR} in construction of the generalized  Hecke algebras for
pseudo reflection groups (see also Ariki-Koike \cite{AK}). Note that
in above connections the operator terms come from pseudo reflection
group, the flatness and $G$-equivariance come from relations in $G$.

Now the Brauer algebra $B_n (\tau )$ has also a natural deformation
$B_n (\tau ,l)$, the BMW algebra. In \cite{Ma1}, the following flat
formal connection supported by  $B_n (\tau)$ was implied :
$$\bar{\Omega}_n =\kappa \sum _{1\leq i<j\leq n} (s_{i,j} -e_{i,j})
\omega_{i,j}.$$ Where $s_{i,j}$ is the $(i,j)$ permutation,
$e_{i,j}$ is the element described by above Figure 1. This
connection $\bar{\Omega}_n$ is flat and $S_n$-equivariant
(Proposition 4 of \cite{Ma1}, see also Proposition 3.1 ). Marin also
proved if $(U, \rho )$ is any representation of $B_n (\tau )$,then
the $S_n
 -$invariant, flat connection $\rho (\bar{\Omega}_n )$ on the bundle
 $M_{S_n} \times U $ induces a flat connection on the quotient
 bundle $M_{S_n} \times _{S_n } U $, whose monodromy representation
 factor through the BMW algebra $B _n (\tau , l)$ for suitable $\tau ,l
 $ (Proposition 4 of \cite{Ma1}, see also Theorem 3.2 ). So the
 connection $\bar{\Omega}_n$ can be seen as the KZ connection for  $B_n (\tau
 )$.

Later we will show (Theorem 5.3 ) the deformation of a simply laced
Brauer algebra $B_{\Gamma}(m)$ to the simply laced BMW algebra can
be  described by a flat connection
\begin{equation}
\bar{\Omega}_{\Gamma}= \kappa \sum_{i\in P} (s_i -e_i  ) \omega_i,
\end{equation}
where $\{ e_i  \}_{i\in P} \subset B_{\Gamma}(m)$ is a set of
semi-idempotents (by a semi-idempotent we mean elements $x$
satisfying $x^2 = \lambda x$ ) in one-to-one correspondence with
$P$, the set of reflection hyperplanes.

 Suppose $\Gamma$ is any finite type Dynkin diagram,  Denote by $W_{\Gamma}$ the Coxeter group of type $\Gamma$. Now we present a bold but reasonable hypothesis about the
general Brauer type algebra $B_{W_{\Gamma}} (\Upsilon)$:
$B_{W_{\Gamma}} (\Upsilon)$ can be deformed to certain BMW type
algebra and the deformation can be described by a nicely shaped KZ
connection $\bar{\Omega}_{\Gamma}$ on $M_{W_{\Gamma}}$.

The most natural form of the connection $\bar{\Omega}_{\Gamma}$ is
as equation $(4)$. When the set of reflections in $W_{\Gamma}$
contain more than one conjugacy class, in view of the connection
$\Omega^{'} _{\Gamma}$ in Equation $(1)$, the KZ connection
associated with $B_{W_{\Gamma}}(\Upsilon)$ should have a more
general form containing more parameters. Thus we make the following
hypothesis
about $B_{W_{\Gamma}}(\Upsilon)$.\\

\nd{\bf Hypothesis 1.} { \it $B_{W_{\Gamma}}(\Upsilon)$ contains the
group algebra $\mathbb{C} W_{\Gamma}$ and a set of semi-idempotents
$\{ e_i \}_{i\in P}$, such that the formal connection  $
\bar{\Omega} _{\Gamma}= \sum _{i\in P} \kappa_i (s_i -e_i ) \omega
_i$  is flat and $W_{\Gamma}$-invariant, where $\{ k_i \}_{i\in P}$
is a set of constant numbers such that $k_i =k_j $ if $i\sim
j$. More over, $B_{W_{\Gamma}}(\Upsilon)$ is generated by $W_{\Gamma} \cup \{ e_i \}_{i\in P}. $   }\\

More generally for a pseudo reflection group $G$, we assume the
algebra $B_G (\Upsilon)$ should also contain $\mathbb{C} G$ and a
set of special elements $\{ e_i \}_{i\in P}$, such that the formal
connection $\bar{\Omega}_G = \sum _{i\in P} (\sum _{s: i(s) =i }
\mu_s s -e_i ) \omega _i$ is flat and $G$-invariant. The shape of
$\bar{\Omega} _G$ is also inspired by the generalized
Lawrence-Krammer representation of $A_G$ defined later.

By Theorem 3.1 in Kohno \cite{Ko1}, we can derive some algebraic
relations between $R \cup \{ e_i \}_{i\in P}$  from flatness and
$G$-invariance of $\bar{\Omega}_G$.  But these relations are not
enough to produce the Brauer type algebra we want.

  Annother common feature of simply laced BMW algebras as proved in Cohen-Gijsbers-Wales \cite{CGW1} is that they all contain
the generalized  Lawrence-Krammer representations of simply laced
Artin groups introduced by  Cohen-Wales \cite{CW} and Digne
\cite{Di}.  Recently Marin \cite{Ma2} introduced a generalized
Lawrence-Krammer representations of $A_G$ for any complex reflection
group $G$ (in this paper we use the phrase 'complex reflection group
' to denote those pseudo reflection groups all of whose pseudo
reflections have degree two). In section 4 we introduce a slightly
further generalization of the Lawrence-Krammer representation to
$A_G$ for each pseudo reflection group G. The idea is as follows.
Let $V_G = \mathbb{C} \{ v_i \} _{i\in P }$ be a vector space with a
basis in one-to-one correspondence with $\{H_i \} _{i\in P}$. Action
of $G$ on $\{H_i \} _{i\in P}$ induces a natural representation
$\iota : G \rightarrow End(V_G )$.

From Marin's work \cite{Ma3} we see that the simply laced type
Lawrence-Krammer representations can be described by certain flat,
$G-$invariant connection
\begin{equation}
\Omega _{LK} =\kappa \sum _{i\in P} (\iota ( s_i ) -p_i )\omega_i
\end{equation}
on the bundle $M_G \times V_G$. Where $s_i \in R$ is the unique
pseudo reflection having $H_i$ as its reflection hyperplane. We
observe that for any $i$ the map $p_i \in End(V_G)$ is a projector
to the line $\mathbb{C} v_i \subset V_G$. It inspires us to consider
a special kind of connection  on the bundle $M_G \times V_G$ for any
pseudo reflection group $G$:
$$\Omega _{LK} = \sum _{i\in P } \left(\sum _{s:\, i(s) =i } \mu _s \iota
(s) - p_i \right) \omega _i$$
 where $\mu_s$'s are constants such that $\mu_{s}=\mu_{s^{'}}$ if $s\sim s^{'}$. And $p_i \in
End(V_G )$ is a projector to the line $\mathbb{C}v_i \subset V_G$
for any $i$. Explicitly suppose $p_i (v_j) =  \alpha _{i,j} v_i
(j\neq i)$ and $p_i (v_i)=m_i v_i$. Then we have\\

\nd { \bf Theorem 4.2 } {\it  The connection $\Omega _{LK}$ is flat
and $G$-equivariant
 if and only if the following two conditions hold{\rm :} $a)$ $\alpha _{i,j} = \sum _{s:
\iota(s)(v_j ) =v_i} \mu _s $; $b)$ $m_i = m_j  $ if $i\sim j$ }.
\\

When $\Omega _{LK}$ satisfy the conditions in Theorem 4.2, it
induces a flat connection $\bar{\Omega}_{LK}$ on the quotient bundle
$M_G \times _G V_G$. We define the generalized Lawrence-Krammer
representation  of $A_G$  as the monodromy representation of
$\bar{\Omega}_{LK}$. When $G$ is a complex reflection group, and
$\mu _s =1$ for all $s$, the connection $\Omega _{LK} $ becomes the
flat connection of Marin £¬ \cite{Ma2}.

 Now suppose $\{ p_i \}_{i\in R} \subset
End(V_G )$ satisfy conditions in Theorem 4.2. It is proved in
Cohen-Gijsbers-Wales \cite{CGW1}  that every simply laced BMW
algebra contain a generalized Lawrence-Krammer representation, just
as the case of braid groups  in Zinno \cite{Z}.    This fact can be
explained in infinitesimal level in the following sense.   Defining
a map $\phi : W_{\Gamma} \cup \{ e_i \}_{i\in P} \rightarrow End(V_G
)$ by $\phi (w) = \iota (w) $ for $w\in W_{\Gamma}$; $\phi (e_i)
=p_i $ for $i\in P$, then  $\phi$ can be extended  to a
representation $B_{\Gamma}(\tau )\rightarrow End(V_G).$

 Regarding to these facts we make another hypothesis about $B_G
(\Upsilon)$:\\

 \nd{\bf Hypothesis 2.} {\it For any pseudo reflection group $G$, the map $w\mapsto \iota (w)$ for $w\in G$;  $  e_i \mapsto
 p_i $ for $i\in P$ can be extended to a representation  $B_G
 (\Upsilon)$. Where we suppose $P_i$ satisfy conditions in above Theorem 4.2. } \\

Now we search if there exist suitable relations between $G$ and $\{
e_i \}_{i\in P}$
 such that the resulted algebra $B_G (\Upsilon)$ satisfy
 Hypotheses
 1 and 2.  As a result we find there do exists one,  the algebra defined in Definition 1.1 satisfy these
 two conditions quite nicely.

\nd{\bf Proposition 5.1} {\it The connection  $ \bar{\Omega}_G =
\sum _{i\in P} (\sum _{s: i(s) =i } \mu_s T_s -e_i ) \omega _i$ are
flat and $G$-invariant. Where $T_s$, $e_i$ are as in Definition 1.1.
}

\nd{\bf Theorem 5.2 } {\it Using notations in section 4. The map
$w\mapsto \iota (w) $, $e_i \mapsto p_i $ extends a representation
$B_G (\Upsilon) \rightarrow End(V_G )$.  }

  In fact we believe it is the best choice. There are
 two other slightly different choices: take
 off the relation $(1)^{'}$ or weaken relation $(6)$.  In the last
 section we explain some reason of choosing Definition 5.1.

 In section 8 we
 show that there exist canonical presentations for $B_G (\Upsilon)$ when $G$ is a finite Coxeter group or a type $G(m,1,n)$ pseudo reflection group.
 (Definition 8.1, Definition 8.2, Theorem 8.4 ). They
  can be seen  a generalization of the presentation
 for simply laced Brauer algebras in \cite{CFW}.  Definition 8.2 can
 be naturally generalized to the cases when $G$ is an infinite type
 Coxeter group (Remark 8.1).
In a canonical presentation each node $i$ of the Dynkin diagram
 corresponds to a pair of generators $\{ s_i , e_i \}$.  Thus in
 cases of $G(m,1,n)$ type we have one more generator $e_0$ than in the
  canonical presentation of cyclotomic Brauer algebras $\mathcal {B}_{m,n}(\sigma)$. It is this
  new generator $e_0$ that making  $B_G (\Upsilon)$ slightly larger than $\mathcal
  {B}_{m,n}(\sigma)$ (Theorem 9.1). Through canonical presentations we see immediately that $B_G
(\Upsilon)$ coincides with the simply
 laced Brauer algebra of \cite{CGW1} if $G$ is a simply laced type Coxeter
 group. These canonical presentations may be helpful to define new BMW type
 algebras, which will be discussed in a future paper.

 Section 6 and 7 are devoted to show these new algebras for non simply laced Dynkin diagrams are indeed interesting objects by
 finding some nice algebraic properties of them.Concretely they
 are: $\bf{(SEM)}$ semisimple for generic parameters; $\bf{(CEL)}$ having
cellular structures; $\bf{(DEF)}$ deformability and $\bf{(STA)}$
dimension stability (having the same dimension for any parameters
$\Upsilon$ ). Because of limitations of spaces in this paper we
 only study in detail the cases when $G$ is a dihedral group or the
 $H_3$ type Coxeter group.

When $G$ is one of above mentioned cases, we prove that $B_G
(\Upsilon)$ satisfies $\bf{(SEM)}$, $\bf{(CEL)}$ and $\bf{(STA)}$
and write down the condition for $B_G (\Upsilon)$ to be semisimple .
Through the study of the $H_3$ case we find for the $H_3$ type Artin
group three new 15 dimensional irreducible representations except
for the generalized Lawrence-Krammer representation, and one new  5
dimensional irreducible representation. All of these representations
have clear combinatorial meaning, they are related to two kinds of
natural actions of $W_{H_3}$ on certain sets. We believe that the
existence of the KZ connections supports the property
$\bf{(DEF)}$ for every $B_G (\Upsilon)$.\\

\nd {\bf Acknowledgements} I would like to thank Toshitake
 Kohno for teaching me KZ equations, and thank  Susumu Ariki,  Sen Hu,   Hebing Rui and Bin Xu
 for many beneficial communications and advices. Especially thanks
 Ivan Marin for pointing out several mistakes in the original
 manuscript. Also thank  Jie Wu for his invitation to NUS in Dec 2008, this work was partially done during
 that stay.

\section{Preliminaries }
  \subsection{Brauer type algebras and BMW type algebras }
  Brauer algebra $B_n (\tau )$ is a graphic algebra
 in the sense that it has a basis
consisting of elements presented by graphs, and the relations
between them can be described through graphs.  $B_n (\tau ) $ has a
canonical presentation with generators  $s_1 , \cdots ,s_{n-1 }$,
$e_1 ,\cdots ,e_{n-1 }$ and relations listed in table 1. $B_n (\tau
)$ has a natural deformation discovered by Birman, Murakami, Wenzl
which are now called BMW algebras \cite{BW} \cite{Mu}. These
algebras support a Markov trace which gives the Kauffman polynomial
invariants of Links.  We denote these BMW algebras as $B _n (\tau ,l
)$.  Where $l$ is a parameter of deformation.  There is $B_n (\tau)
\cong B_n (\tau , 1 ).$ We list generators and relations of $B_n
(\tau ) $ and $B _n (\tau ,l )$ in the following table according to
\cite{CGW1}. Where $m=\frac {l-l^{-1}}{1-\tau}$.

The structure of Brauer algebras and BMW algebras are studied
extensively in last 20 years. See for example \cite{W} \cite{RH}.
They have the following basic properties.

\nd{\bf Theorem (Wenzl)} Let the ground ring be a field of character
0, then $B_n (\tau )$ is semisimple if and only if $\tau \notin
\mathbb{Z}$ or $\tau \in \mathbb{Z}$ and $\tau >n .$\\
$$ TABLE\  1.\ Presentation\ for\  B_n (\tau ). $$
\begin{tabular}{|l|l|l|}

\hline

   & $B_n (\tau )$  & $B _n (\tau,l)$ \\
\hline

  Generators    & $s_1 $,$\cdots $,$s_{n-1}$;$e_1$ ,$\cdots$ ,$e_{n-1}$ & $X_1$ , $\cdots $,$X_{n-1}$; $E_1$ ,$\cdots$ ,$E_{n-1}$ \\
\hline

  Relations   &  $s_i s_{i+1} s_i =s_{i+1} s_i s_{i+1}$    & $X_i X_{i+1} X_{i} =X_{i+1} X_{i} X_{i+1}$\\
&$for 1\leq i\leq n-2 $; &$for 1\leq i\leq n-2$  ;\\

& $s_i s_{i-1} e_i = e_{i-1} s_i s_{i-1}  $
& $X_i X_{i-1} E_i = E_{i-1} X_i X_{i-1} $\\
& $for 2\leq i\leq n-1$;  &$for 2\leq i\leq n-1$;\\

 & $s_i s_{i+1} e_i = e_{i+1} s_i s_{i+1}  $
& $X_i X_{i+1} E_i = E_{i+1} X_i X_{i+1} $\\
&$for 1\leq i\leq n-2$; &$for 1\leq i\leq n-2$;\\
& $s_i s_j =s_j s_i$ ,$|i-j|\geq 2 $;& $X_i X_j =X_j X_i ,|i-j| \geq 2$ ;\\
 &$s_i ^2 =1 ,for\  all\  i $;& $l(X_i ^2 +mX_i -1) =mE_i , for\  all\
i $;\\
 & $s_i e_i =e_i ,for\ all\ i $;&$ X_i E_i = l^{-1} E_i ,for\ all\ i $;\\
& $e_i s_{i+1} e_{i} = e_{i}, 1\leq i \leq n-2 $; &$ E_i X_{i+1}
E_{i} = l E_i ,1\leq i \leq n-2 $;
\\
&$ e_i s_{i-1} e_i =e_i , 2\leq i \leq n-1$ ; &$ E_i X_{i-1} E_{i} =
lE_{i} ,2\leq i\leq n-1 $; \\
&$s_i e_j = e_j s_i , |i-j|>1 $; & $X_i E_j =E_j X_i ,|i-j|>1 $;\\
&$e_i ^2 = \tau e_i ,for\  all\  i.$ &$ E_i ^2 =\tau E_i $
. \\

\hline
\end{tabular}\\

Semisimplicity  condition for any groundrings is obtained by Rui
$\cite{RH}$. Many algebras related to Lie theory have cellular
structures in the sense of Graham and Lehrer $\cite{GL}$. We recall
the definition of a cellular structure (cellular algebra). In the
same paper Graham and Lehrer proved Brauer algebras support cellular
structures. Similar result for BMW algebras are proved by Xi
$\cite{Xi2}$.

\nd{\bf Definition (Graham, Lehrer)[15]} A cellular algebra over $R$
is an associative algebra $A$, together with cell datum $(\Lambda
,M,C,*)$ where
\begin{itemize}
  \item (C1) $\Lambda $ is a partially ordered set and for each $\lambda \in
\Lambda$ ,$M(\lambda)$ is a finite set such that$ C:\cap _{\lambda
\in \Lambda} M(\lambda ) \times M(\lambda ) \rightarrow A $ is an
injective map with image an R-basis of A.
 \item (C2) If $\lambda \in \Lambda $ and $S,T\in M(\lambda )$,
  write $C(S,T)= C^{\lambda } _{S,T} \in A$. Then $*$ is an
  $R$-linear anti-involution of A such that $*(C^{\lambda } _{S,T} )  =C^{\lambda }
  _{T,S}$.
\item (C3) If $\lambda \in \Lambda $ and $S,T \in M(\lambda )$
  then for any element $a\in A$ we have

 $aC^{\lambda } _{S,T} \equiv \sum _{S^{'}\in M(\lambda )} r_a (S^{'} ,S) C^{\lambda} _{S^{'},T} (modA(<\lambda))
   $

   Where $r_a (S^{'} ,S)\in R$ is independent of $T$ and where $A(<\lambda
   )$ is the R-submodule of $A$ generated by $\{ C^{\mu } _{S^{''} ,T^{''} } | \mu < \lambda ; S^{''}, T^{''} \in M(\mu )
   \}.$.

\end{itemize}

The cyclotomic Brauer algebras $\mathscr{B} _{m,n}(\delta )$ of
H\"{a}ring-Oldenburg has the following presentation. (borrowed from
$\cite{RX}$)
\begin{defi}
 The algebra $\mathscr{B} _{m,n}(\delta )$ is generated by a set
 $\{ s_i ,e_i \}_{1\leq i< n} \cup \{ t_j \} _{1\leq j\leq n}$ with
 the following  relations.
\begin{align*}
&a) s_i ^2 =1 ,for\  1\leq i\leq n. & k)& e_i s_i =e_i =s_i e_i ,for\  1\leq i \leq n-1. \\
&b) s_i s_j =s_j s_i ,if\  |i-j|>1. & l)& s_i e_{i+1} s_i =s_{i+1} e_i, for\  1\leq i\leq n-2. \\
&c) s_i s_{i+1} s_i = s_{i+1} s_i s_{i+1}, for\ 1\leq i<n-1. & m)& e_{i+1} e_i s_{i+1} =e_{i+1} s_i, \\
&d) s_i t_j =t_j s_i, if\ j\neq i,i+1.  & &   for 1\leq i\leq n-2. \\
&e) e_i ^2 =\delta _0 e_i , for\ 1\leq i<n.      &n)& e_i e_j e_i =e_i ,if\ |i-j|=1.     \\
&f) s_i e_j =e_j s_i ,if\ |i-j|>1.      &o)& e_i t_i t_{i+1} =e_i =t_i t_{i+1} e_i ,    \\
&g) e_i e_j =e_j e_i ,if\ |i-j|>1.       &\ & \ \ \ \ \ \ for \ 1\neq i<n.     \\
&h) e_i t_j =t_j e_i ,if\ j\neq i,i+1.     &p)& e_i t_i ^a e_i =\delta _a e_i , for\ 1\leq a\leq m-1     \\
&i) t_i t_j =t_j t_i ,for \ 1\leq i,j\leq n.      &\ & \ \ \ \ \ 1\leq i\leq n-1.     \\
&j) s_i t_i =t_{i+1} s_i ,for 1\leq i<n . &q)& t_i ^m=1,for\ 1\leq
i\leq n.
\end{align*}
\end{defi}

The subset of generators $\{s_i \}_{1\leq i<n} \cup \{ t_j \}_{1\leq
j\leq n}$ together with relations (a),(b),(c),(d),(i),   (j),(q)
generate the cyclotomic reflection group of type $G(m,1,n)$ whose
group algebra is imbedded in $\mathscr{B} _{m,n}(\delta )$. The
original paper $\cite{Ha}$ define more complicated cyclotomic BMW
algebra, where the generators and relations can be represented by
graphs also. These algebras have many properties parallel with
Brauer algebras. In  $\cite{RX}$, the authors proved they are
semisimple and classified their irreducible representations under
certain generic conditions. By Goodman in $\cite{Go}$ and by Yu in
$\cite{Yu}$ independently, $\mathscr{B} _{m,n}(\delta )$ are shown
to have cellular structures.

  Finite type simply-laced Dynkin diagram consists of  ADE type Dynkin
diangrams. For every such Dynkin diagram $\Gamma$, the following
table are presentation for algebra $B _{\Gamma } (\tau )$ and
algebra
 $B_{\Gamma } (\tau , l)$ defined in $\cite{CGW1}$. When $\Gamma $
 is $A_{n-1}$, it is straightforward to see they coincide with $B_n (\tau
 )$ and $B_n( \tau ,l)$ respectively. Let $I$ be the set of nodes of $\Gamma$. When $i, j \in I$ are
connected by an edge we write $i\sim j$. Otherwise we write $i\nsim
j$. Set $m=\frac {l-l^{-1}}{1-\tau }$.

The simply laced  Brauer algebras have no graph to representing
their elements any more, but they have almost all important
algebraic properties of Brauer algebras. In $\cite{CFW}$ the authors
proved when $\Gamma$ is finite ADE type, these algebras $B_{\Gamma }
(\tau )$ are free module over $\mathbb{Z}[\tau ^{\pm 1} ]$, and be
semisimple after tensored with $\mathbb{Q} (\tau )$.
$$TABLE\ 2.\ Presentation\ for\ B_{\Gamma }
(\tau ).$$
\begin{tabular}{|l|l|l|}

\hline

   & $B_{\Gamma } (\tau )$  & $B _{\Gamma } (\tau,l)$ \\
\hline

  Generators    & $s_i  (i\in I ) $; $e_i (i\in I) $& $X_i (i\in I)$ ; $E_i (i\in I)$ \\
\hline

  Relations   &  $s_i s_j s_i =s_j s_i s_j$, $if\ i\sim j $  ;  & $X_i X_j X_{i} =X_j X_{i} X_j$, $if\ i\sim j$;
\\
 & $s_i s_j e_i = e_j s_i s_j $ if $i \sim j$; & $X_i X_j E_i = E_j X_i X_j
 $ if $i \sim j$;\\
 & $s_i s_j =s_j s_i$ ,$if\ i\nsim j $;& $X_i X_j =X_j X_i ,if\ i\nsim j $ ;\\
 &$s_i ^2 =1 ,for\  all\  i $;& $l(X_i ^2 +mX_i -1) =mE_i , for\  all\
i $;\\
 & $s_i e_i =e_i ,for\ all\ i $;&$ X_i E_i = l^{-1} E_i ,for\ all\ i $;\\
& $e_i s_{j} e_{i} = e_{i}, if\ i\sim j $; &$ E_i X_{j} E_{i} = l
E_i ,if i\sim j $;
\\

&$s_i e_j = e_j s_i , if\ i\nsim j $; & $X_i E_j =E_j X_i ,if i\nsim j $;\\
&$e_i ^2 = \tau e_i ,for\  all\  i.$ &$ E_i ^2 =\tau E_i $
. \\

  \hline
\end{tabular}\\

\subsection{Pseudo reflection groups, Complex braid groups and Hecke algebras}
Let $V$ be a complex linear space. An element $s$ in $GL(V)$ is
called a pseudo reflection if it can be presented as $diag(\xi
,1,\cdots ,1)$ under some basis of $V$, where $\xi$ is a root of
unit. We call $\xi$ as the exceptional eigenvalue of $s$.  If $\xi$
is $-1$ then $s$ is simply called a reflection. A finite group
$G\subset GL(V)$ is called a pseudo reflection group if it is
generated by pseudo reflections. If $G$ is generated by reflections
then we call it a complex reflection group.  When $V$ is an
irreducible representation of $G$, $(V,G)$ is called an irreducible
pseudo reflection group. Every pseudo reflection group is isomorphic
to direct product of some irreducible factors. Isomorphism class of
irreducible pseudo reflection groups are classified by Shephard-Todd
\cite{ST}. They consists of an infinite family $\{\  G(m,p,n)\  \} $
( $n\leq 1 $,$m\leq 2$ ,$p|m $ ) and 34 exceptional ones.

For a pseudo reflection group $(V, G)$ we assume notations in
section 1. Denote $\pi _1 (M_G)$ as $P_G$, then there is an exact
sequence:  $1\rightarrow P_G \rightarrow A_G \rightarrow G
\rightarrow 1 .$

By Ariki, Koike in \cite{AK} and by Broue-Malle-Rouquier in
\cite{BMR}, there exists a Hecke algebra $H _G (\bar {\lambda})$
associated with any pseudo reflection group $G$, where $\bar
{\lambda}$ is a set of parameters. The Hecke algebra $H _G (\bar
{\lambda})$ is a quotient algebra of the group algebra $\mathbb{C}
A_G $. For most $G$'s, we have  $\dim H _G (\bar {\lambda})=  |G|$,
and this relation is a conjecture for other cases.  For some $G$'s
and for generic $\bar {\lambda}$, $H _G (\bar {\lambda})$ is a
semisimple algebra whose irreducible representations are in one to
one correspondence with those of $G$ in a natural way. This
correspondence can be described by the following KZ connection.

   Suppose $\{ \mu _s \} _{s\in R
 }$ is a set of constants satisfying the condition:
$\mu _{s_1 } =\mu _{s_2} $ if $s_1 $ is conjugate to $s_2$. Here for
simplicity we choose a connection with slightly different appearance
from \cite{BMR}.

\begin{prop}[Brou\'e -Malle-Rouquier \cite{BMR}]
The formal connection  $$\Omega _G  =  \sum _{i\in P } (\sum _{s\in
R,i(s)=i }\mu _s s )\omega _i $$ on $M_G \times \mathbb{C} G$is flat
and $G$- invariant.
\end{prop}

Now suppose $\rho : G\rightarrow GL(U)$ is a representation of $G$
on a complex linear space $U$.  The group $G$ acts on the bundle
$M_G \times U$ as: $g(p ,\ u )= (g\cdot p , \rho (g)(u))$ for $g\in
G$, $p\in M_G$ and $u\in U$. The quotient space $M_G \times U / G$
become a linear bundle over $M_G /G$ naturally, and it will be
denoted as $M_G \times _G U$.  Now suppose $\Omega =\sum _{i\in P}
X_i \omega _i $ is a connection on $M_G \times U$, where $X_i \in
End (U)$ for any $i.$ Here is the condition for $\Omega$ to induce a
connection on $M_G \times _G U$. See section 4 of \cite{BMR} for
some backgrounds about connections.

\begin{prop}
The connection $\Omega$ induce a connection on $M_G \times _G U$ if
:

$\rho (w) X_i \rho (w) ^{-1} = X_{w(i)}$ for any $w\in G$ and $i\in
P$.
\end{prop}

When the condition in above proposition is satisfied, we call
$\Omega$ as a $G$-invariant connection. Suppose $(E,\rho )$ is a
linear representation of $G$ . By above proposition $$\rho (\Omega
_G ) = \kappa \sum _{i\in P } (\sum _{s\in R,i(s)=i }) \mu _s \rho
(s) \omega _i $$ defines a flat connection on the bundle $M_G \times
E$. It induces a flat connection $ \bar {\Omega } _{\rho }$ on the
quotient bundle $M_G \times _G E $ because of $G$-invariance of
$\rho (\Omega _G )$. By taking monodromy a family of representations
of $A_G $ parameterized by ($\kappa $, $\mu _s $) are obtained. It
is proved in \cite{BMR} that for generic $\kappa$ these monodromy
representations factor through $H _G (\bar {\lambda})$ for suitable
$\bar {\lambda} $.

The following theorem from Marin \cite{Ma2}(Theorem 2.9 ) will be
used in the following  sections. Let $G\subset GL(V)$ be a pseudo
reflection group. Let $\mathcal {A}$, $P$, $\omega_i$, $M_G$, $P_G$,
$A_G$ be defined as in Section 1.  Suppose $I\subset V$ be a complex
line. The maximal parabolic subgroup $G_0$ of $G$ associated with
$I$ is the subgroup of $G$ formed by elements which stabilize $I$
pointwise. By Steinberg's theorem, $G_0$ is generated by reflections
$R_0$ of $G$ whose reflecting hyperplane contains $I$. We set
$\mathcal {A}_0 =\{ H_i \in \mathcal {A} | I\subset H_i \}$, and
$P_0 =\{ i\in P | I\subset H_i \}$. Let $M_0 = V-\cup _{i\in P_0 }
H_i $. Since $G_0$ is a pseudo reflection group, it has associated
braid group $A_{G_0}$ and pure braid group $P_{G_0}$. It is clear we
have identifications: $P_{G_0} \cong \pi_1 (M_0)$, $A_{G_0} \cong
\pi_1 (M_0 / G_0 )$. Following \cite{BMR}, $P_{G_0}$ and $A_{G_0}$
can be imbedded into $P_G$ and $A_G$ in the following way, whose
image are called maximal parabolic subgroup of $P_G$, respectively
$A_{G}$.

We endow $V$ with a $G$-invariant unitary form and denote the
associated norm as $\parallel \ \ \parallel$. Let $x_1 \in I$ such
that $x_1 \notin H $ for any $H\in \mathcal {A}\setminus \mathcal
{A}_0$. There exists $\epsilon >0$ such that, for all $x\in V$ with
$\parallel x-x_1 \parallel \leq \epsilon$ we have $x\notin H$ for
all $H\in \mathcal {A}\setminus \mathcal {A}_0$. Let $D=\{ x\in V |
\parallel x-x_1 \parallel \leq \epsilon \}$. It is easy to see the
natural morphism $\pi_1 (M_G \cap D ) \rightarrow \pi_1 (M_0 )$ is
an isomorphism, hence the natural inclusion $\pi_1 (M_G \cap D
)\rightarrow \pi_1 (M_G)$ defines an embedding $P_{G_0} \rightarrow
P_G$. Since $D$ is setwise stabilized by $G_0$, this embedding
extends to an embedding $A_{G_0} \rightarrow A_G$. It is proved in
\cite{BMR} that such embeddings are well-defined up to
$P_G$-conjugation.

 Now suppose on a bundle $M_G \times E$
there is a flat connection $\Omega =\kappa \sum_{i\in P} X_i
\omega_i $. Denote the monodromy representation of $P_G$ resulted
from $\Omega$ as $\rho$. If $\Omega$ is $G$-invariant, denote the
monodromy representation of $A_G$ resulted from $\Omega$ as
$\tilde{\rho}$. Looking $P_{G_0}$, $A_{G_0}$ as parabolic subgroups
of $P_G$, $A_G$, we obtain representations of $P_{G_0}$ and
$A_{G_0}$ by restriction of $\rho$ and $\tilde{\rho}$ respectively.
On the other hand, we define a connection on $M_0$: $\Omega _0
=\kappa \sum_{i: I\subset H_i } X_i \omega_i $. An simple discussion
by using Theorem 3.1 of Section 3 shows $\Omega_0 $ is also flat. We
denote the monodromy representation of $P_{G_0}$ resulted from
$\Omega_0$ as $\rho_0$. When $\Omega$ is $G$-invariant, then
$\Omega_0$ is $G_0$-invariant. In these cases we denote the
monodromy representation of $A_{G_0 }$ resulted from $\Omega_0$ as
$\tilde{\rho}_0$. The following theorem is proved in Marin
\cite{Ma2} (Theorem 2.9).

\begin{thm}
For generic $\kappa$, the $P_{G_0}$ representation $\rho_0$ is
isomorphic to the restriction of $\rho$. When $\Omega$ is
$G$-invariant, the $A_{G_0}$ representation $\tilde{\rho}_0$ is
isomorphic to the restriction of $\tilde{\rho}$.
\end{thm}

\section{Flat connections for BMW algebras }

We begin with some knowledge for hyperplane arrangements. Let $E$ be
a complex linear space. An hyperplane arrangement (or arrangement
simply ) in $E$ means a finite set of hyperplanes contained in $E$.
Let $\mathscr{A} =\{ H_i \} _{i\in I}$ be an arrangement in $E$, we
denote the complementary space $ E- \cup _{i\in I} H_i $ as
$M_{\mathscr{A}}$.  Intersection of any subset of $\mathscr{A}$ is
called an edge. If $L$ is an edge of $\mathscr{A}$, define
$\mathscr{A} _L = \{ H_i \in \mathscr{A} | L\subset H_i \} =\{H_i \}
\ and\ I_L = \{ i\in I | L\subset H_i  \} .$

For every $i\in I$, chose a linear form $f _i $ with kernel $H_i $.
Set $\omega _i = d\log f_i $, which is a holomorphic closed 1-form
on $M_{\mathscr{A}}$.  Consider the formal connection $\Omega =
\kappa \sum _{i\in I} X_i \omega _i $.  Here $X_i $ are linear
operators to be determined. When we take $X_i $'s as endomorphisms
of some linear space $E$, then $\Omega $ is realized as a connection
on the bundle $M_{\mathscr{A}} \times E$.  We have the following
theorem of Kohno.

\begin{thm}[Kohno \cite{Ko1}]
 The formal connection $\Omega $ is flat if and only if:

 $[ X_i ,\sum _{j\in I_L } X_j ] =0$ for any codimension  2  edge  $L$  of  $\mathscr{A}$,and for any
$i\in I_L$. Where $[A,B]$ means $AB-BA$.
\end{thm}

The following lemma can be proved directly by using graphs.

\begin{lem}
In the Brauer algebra $B _n (\tau )$, let $s_{i,j} \in S_n \subset B
_n (\tau )$ be $(i,j)$ permutation., let $e_{i,j}$ be as in
introduction. we have
\begin{itemize}
\item[(1)] $e_{i,j} s_{k,l} =s_{k,l} e_{i,j} $ if $ \{i,j\} \cap
\{k,l\}=\varnothing ;$
\item[(2)] $e_{i,j} e_{k,l} =e_{k,l} e_{i,j}$
if $ \{i,j\} \cap \{k,l\}=\varnothing ;$
\item[(3)] $e_{i,j} = e_{j,i};$
\item[(4)] $e_{i,j} e_{i,k} = s_{j,k} e_{i,k} = e_{i,j} s_{j,k}$ ,for
any different $i,j,k$;
\item[(5)] $ e_{i,j} ^2 = \tau e_{i,j}$ , for
any $ i\neq j$;
\item[(6)] $ s_{i,j} e_{j,k} = e_{i,k} s_{i,j}$.
\end{itemize}
\end{lem}

 For $1\leq i<j\leq n-1 $, define $\omega _{i,j}= d(z_i -z_j )/(z_i -z_j
)$. Consider the formal connection $\bar{\Omega} _n = \kappa \sum
_{i<j } (s_{i,j} -e_{i,j}) \omega_{i,j}$. We have

\begin{prop}[Marin\cite{Ma1}]
The formal connection $\bar{\Omega} _n$ is flat and $S_n $
invariant.

\end{prop}

\begin{proof}
We certify $\bar{\Omega} _n$ satisfies conditions of theorem 3.1.
For the arrangement $\mathscr{A} _n $,there are then following two
type of codimension 2 edges

Case 1. $L= H_{i,j} \cap H_{k,l}$, $\{i,j\}\cap
\{k,l\}=\varnothing$.

Whence $\mathscr{A} _L =\{H_{i,j} ,H_{k,l}\}$. Now we have $s_{a,b}
s_{c,d} =s_{c,d} s_{a,b}$  and $e_{a,b} e_{c,d} =e_{c,d} e_{a,b}$ if
$\{a,b \}\cap \{c,d \}=\varnothing .$ They are most easily seen by
using graphs. so $[s_{i,j} -e_{i,j} ,s_{k,l} -e_{k,l}]=0.$ Which
gives $[s_{i,j} -e_{i,j} ,s_{i,j} -e_{i,j} +s_{k,l} -e_{k,l}]=0=
 [s_{k,l} -e_{k,l} ,s_{i,j} -e_{i,j} +s_{k,l} -e_{k,l}].$

Case 2. $L=H_{i,j} \cap H_{j,k}$, where $i,j,k$ are different. In
this case $\mathscr{A} _L = \{ H_{i,j} ,H_{j,k} ,H_{i,k} \}$,
\begin{align*}
&[s_{i,j} - e_{i,j} ,  s_{i,k} -e_{i,k} +s_{j,k} -e_{j,k}] \\
 =&[s_{i,j},s_{i,k} +s_{j,k}] +(-e_{i,j} s_{i,k}
+e_{i,j}
e_{j,k})+(s_{i,k}e_{i,j} -e_{j,k} e_{i,j} ) \\
   +&
 (-e_{i,j} s_{j,k} +e_{i,j} e_{i,k} ) + (s_{j,k} e_{i,j}
-e_{i,k}e_{i,j}) +[s_{i,j} ,-e_{i,k}-e_{j,k}]  \\
=& (-e_{i,j} s_{i,k} +e_{i,j}e_{j,k})+(s_{i,k}e_{i,j} -e_{j,k}
e_{i,j} ) \\
 +&(-e_{i,j} s_{j,k} +e_{i,j} e_{i,k} ) + (s_{j,k}
e_{i,j}
-e_{i,k}e_{i,j})  +[s_{i,j} ,-e_{i,k}-e_{j,k}]  \\
 =& (-e_{i,j}s_{i,k} +e_{i,j}
e_{j,k})+(s_{i,k}e_{i,j} -e_{j,k} e_{i,j} )
 +(-e_{i,j} s_{j,k} +e_{i,j} e_{i,k} ) \\
  +&(s_{j,k} e_{i,j}-e_{i,k}e_{i,j})
 =0.
\end{align*}

The second equality is because $s_{i,j} s_{i,k} +s_{i,j} s_{j,k} =
s_{j,k} s_{i,j} +s_{i,k} s_{i,j}. $ For the third equality use Lemma
3.1, (6). For the fourth equality use lemma 3.1, (4). $G$-invariance
of $\bar{\Omega} _n$ is evident.

\end{proof}

Let $(E,\rho)$ be a finite dimensional representation of $B _n (m
)$. By proposition 3.1, the connection
$$ \rho (\bar{\Omega} _n)= \kappa \sum _{i<j }
( \rho (s_{i,j} ) -\rho (e_{i,j}) ) \omega_{i,j}$$ induce a flat
connection on the bundle $Y_n \times _{S_n} E $, which is a linear
bundle on $X_n$. Denote the resulted monodromy representation of
$\pi _1 ( Y_n / S_n ) \cong B_n $ as $\bar{\rho}$. The following
theorem can be found in \cite{Ma1}.

\begin{thm}[Marin\cite{Ma1}]
For generic $\kappa$, the monodromy representations $\bar{\rho}$ of
$B_n$ constructed above factor through $B _n (\tau , l)$, for $\tau
= \frac{q^{1-m} - q^{m-1} + q^{-1} -q}{q^{-1} -q}$, $l=q^{m-1} $.
Where $q= \exp \kappa \pi \sqrt{-1}$.
\end{thm}

\section{Generalized Lawrence-Krammer Representations }

The Lawrence-Krammer representations and their generalizations play
a significant role in the theory of braid groups and Artin groups.
See Krammer \cite{Kr}, Bigelow \cite{Bi}, Cohen and Wales \cite{CW},
Digne \cite{Di},Paris \cite{Pa} and Marin \cite{Ma2} \cite{Ma3}.
Since this paper concentrate on infinitesimal level, we majorly
refer to  \cite{Ma2} \cite{Ma3}.

 Let $V$ be a
$n$-dimensional complex linear space. Let $G\subset U(V)$ be a
complex reflection group. Let $R$ be the set of reflections in $G$.
We use notations in section 2.2.

The generalized LK representations of $A_G$ of Marin are described
by certain flat connections as follows.  First, for every $H_s $, we
have a closed $1$-form $\omega _s$ on $M_G $ as in section 2.2. Then
let $V_G =\mathbb{C}\langle v_s \rangle _{s\in R}$ be a complex
linear space with a basis indexed by $R$. For every pair of elements
$s,u \in R$, define a nonnegative integer $\alpha (s,u) =\# \{ r\in
R | rur=s \} $. Chose a constant $m\in \mathbb{C}$. For any $s\in
R$, define a linear operator $t_s \in GL(V_G )$ as follows:
$$t_s \cdot v_s = m v_s , \ \ t_s \cdot v_u = v_{sus} -\alpha (s,u)v_s \ \ for\ \  s\neq u. $$

Chose another constant $k\in \mathbb{C}$. Define a connection
$\Omega _{K} = \sum _{s\in R} k\cdot t_s \omega _s $ on the trivial
bundle $V_G \times M_G $.

\textbf{Theorem and Definition }(Marin\cite{Ma2}) {\it The
connection $\Omega _{K} $ is flat and $G$-invariant. So it induce a
flat connection $\bar{\Omega }_{K}$ on the quotient bundle $M_G
\times_G V_G $. The generalized LK representation for $A_G $ is
defined as the
monodromy representation of $\bar {\Omega }_{K}$. }\\

We denote the generalized Krammer representation as $(V_G ,\rho
_{\kappa ,m })$. When $(G,V)$ is the reflection group $W_{\Gamma }$
of ADE type,  they were first constructed in \cite{CW} by Cohen,
Wales and by Digne in \cite{Di}.  They are proved to factor through
BMW algebras in \cite{CGW1}.

\begin{thm}[Marin \cite{Ma1} ] The generalized Krammer representation $(V_G , \rho _{\kappa ,m })$factor
through the generalized BMW algebra $B _{\Gamma } (\tau ,l)$ with
$\tau =\frac {q^m -q^{-m} +q^{-1} -q}{q^{-1}-q}$ and $l= q^{-m}$.
Where $q=e^{\kappa \pi \sqrt{-1}}$.

\end{thm}

For later convenience we change notations slightly. For $s\in R$, we
define $p_s : V_G \rightarrow V_G $ by
$$ p_s (v_s )= (1-m_s )v_s ,\ p_s (v_u )= \alpha (s,u ) v_s \ for\ u\neq s.$$

We also define $\iota :G \rightarrow Aut (V_G )$ by $\iota (w) (v_s)
= v_{wuw^{-1}}$.  Then $p_s$ is a projector to the complex line
$\mathbb{C}v_s $. And Marin's flat connection $\Omega _K$ is written
as $\sum _{s\in \mathcal {R}} k\cdot (\iota (s) -p_s ) \omega _s$.\\

\nd{\bf A Further Generalization }  Let $V$ be a $n$-dimensional
complex linear space. Let $G\subset U(V)$ be a finite pseudo
reflection group(not only complex reflection group). We define $P,R,
\mathcal {A}$ for $G$ as in Section 1. For $s\in R$, denote the
reflection hyperplane of $s$ as $H_s$. Define $V_G =\mathbb{C} <v_i
>_{i\in P}$. Since $w(H_v )$ is annother reflection hyperplane for any
$w\in G $ and $v\in P$, there is an action of $G$ on $P$ which
induce a representation $\iota : G\rightarrow Aut (V_G )$.
Explicitly $w(i)$ is defined by $H_{w(i)}= w(H_v )$. For $i\in P$,
let $p_i :V_G \rightarrow V_G $ be a projector to $\mathbb{C} v_i $
which is written as:
$$p_i (v_i )=m_i v_i ,\ p_i (v_j ) =\alpha _{i,j} v_i .$$
As in Section 1 let $\{\mu _s \} _{s\in R }$ be a set of nonzero
constants such that: $\mu _{s_1} =\mu _{s_2 } $ if $ s_1$ is
conjugate to $s_2 $ in $G$. Define a function $i: R\rightarrow P $
such that $H_{i(s)}$ is the reflection hyperplane of $s$ for any
$s\in R$. Consider a connection $\Omega _{LK} $ on the trivial
bundle $M_G \times V_G $ which have the form
$$ \sum _{i\in P} (\sum _{s: i(s)=i } \mu _s \iota (s) -p_i ) \omega _i .$$

\begin{thm}  The connection $\Omega _{LK}$ is flat and $G$-equivariant if
and only if the the following conditions are satisfied:

$(1)$ $m_i =m_j$ if there is $w\in G$ such that $\iota
  (w)(v_i)=v_j$.

  $(2)$ $\alpha _{i,j} =\sum _{s : \iota (s )(v_j )=v_i } \mu _{s} .$

 \end{thm}

 \begin{proof} First we suppose $\Omega _{LK}$ is a flat, $G$-equivariant connection. By Proposition 2.2 we have,
 $$ \iota (w) ( \sum_{s:i(s)=i} \mu_s \iota (s) -p_i) \iota (w)^{-1} = \sum_{s:i(s)=w(i)} \mu_s \iota (s) -p_{w(i)}.$$
 By condition of the set $\{ \mu_s \}_{s\in R}$, above
 identity is equivalent to $\iota (w) p_i \iota (w)^{-1} =
 p_{w(i)},$ which implies $m_i = m_{w(i)} .$

Let $L$ be any codimension 2 edge of the arrangement $\mathscr{A}$.
Let $H_{i_1},\cdots , H_{i_{N}}$ be all the hyperplanes in
$\mathscr{A}$ containing $L$. By theorem 3.1,  flatness of
$\Omega_{LK}$ implies :
\begin{equation}
[ \sum _{s: i(s)=i_a } \mu _s \iota (s) -p_{i_a } , \sum _{v=1} ^{N}
( \sum _{s:i(s) =i_v } \mu _s \iota (s) -p_{i_v } ) ] =0.
\end{equation}

for $1\leq a\leq N $. Without losing generality suppose $a=1$.  It
is equivalent to the following identity because by Proposition 2.1,
the sum of those terms containing no $p_i$ is zero.

\begin{equation}
[ p_{i_1 } , \sum _{v=1} ^{N} ( \sum _{s:i(s) =i_v } \mu _s \iota
(s)- p_{i_v } ) ] + [ \sum _{s: i(s)=i_1 } \mu _s \iota (s)  , \sum
_{v=1} ^{N} p_{i_v } ) ] =0.
\end{equation}

Now for those $s$ such that $ i(s) =i_1$ we have $\{ s(i_1 ), \cdots
, s(i_N ) \} =\{ i_1 ,\cdots ,i_N \}. $  So we have:
\begin{equation}
\begin{split}
 [ \sum _{s: i(s)=i_1 } \mu _s \iota (s)  ,
\sum _{v=1} ^{N} p_{i_v } ) ] &=
\sum _{s: i(s)=i } \mu _s \sum _{v=1} ^{N} (\iota (s) p_{i_v }-p_{i_v }\iota (s) )\\
   &= \sum _{s: i(s)=i } \mu _s \sum _{v=1} ^{N} (  p_{s(i_v )} \iota (s) -p_{i_v }\iota (s)) \\
   &= 0. \\
\end{split}
\end{equation}
So the identity (7) is equivalent to:

\begin{equation}
\begin{split}
  [ p_{i_1 } , \sum _{v=1} ^{N} ( \sum _{s:i(s) =i_v } \mu _s \iota (s)
-p_{i_v } ) ] &= \sum _{v=2} ^{N} \sum _{s:i(s) =i_v } [ p_{i_1 } ,
( \sum _{s:i(s) =i_v } \mu _s \iota (s) -p_{i_v } ) ]\\
&=0.
\end{split}
\end{equation}

This is because $[p_{i_1 } , \iota(s)]=0 $ if $i(s) =i_1 $. After
splitting the Lie bracket in equation (9), the sum of all those
terms mapping to $\mathbb{C} v_{i_u }$ is $p_{i_u } p_{i_1} -\sum
_{s: s(v_{i_1 } ) =v_{i_u }}\mu_s \iota (s) p_{i_1 } $. It must be
0. Chose $s_0 $ such that $s_0 (v_{i_1 }) =v_{i_u }$ if there exist
one, then we have $ p_{i_u } p_{i_1 }= \alpha _{i_u ,i_1 } \iota
(s_0 ) p_{i_1 }$. More over, for any $s$
 such that $s(v_{i_1 } ) =v_{i_u }$  we have $\iota (s) p_{i_1 } = \iota (s_0 )
p_{i_1 }$. Put these identities in equation  $p_{i_u } p_{i_1} -\sum
_{s: s(v_{i_1 } ) =v_{i_u }}\mu_s \iota (s) p_{i_1 } =0 $ , we get
$$(\alpha _{i_u ,i_1 } -\sum _{s: s(v_{i_1 })=v_{i_u }} \mu _s ) \iota (s_0 ) p_{i_1 } =0 .$$

So we have $\alpha _{i_u ,i_1 } =\sum _{s: s(v_{i_1 })=v_{i_u }} \mu
_s .$ If there don't exist such $s_0$, we can prove $\alpha _{i_u
,i_1 } =0 $ similarly.

Now suppose conditions $(1)$ and $(2)$ are satisfied, by the same
arguments we only need to prove above equation $(9)$ to show $\Omega
_{LK}$ is flat. The conditions (2) implies
\begin{equation}
p_i p_j = \sum _{s: s(j)=i } \mu _s \iota (s) p_j , for\ any\ i\neq
j.
\end{equation}
It also implies
\begin{equation}
p_i p_j = \sum _{s: s(j)=i } \mu _s p_i \iota (s)  , for\ any\ i\neq
j.
\end{equation}
since $\iota (s) p_j = \iota (s)p_j \iota (s) ^{-1} \iota (s) = p_i
\iota (s)$ for those $s$ such that $s(j)=i$.  Now the right hand
side of equation $(9)$ can be written as
$$\sum _{v=2} ^N (p_{i_1 } p_{i_v } - \sum _{s: s(i_v ) =i_1 } \mu _s p_{i_1 } \iota (s))+
\sum _{v=2} ^{N} (p_{i_v } p_{i_1 } -\sum _{s: s(i_v )=i_1 } \mu _s
\iota (s) p_{i_v }).$$

So the equation $(6)$ is true and it implies flatness of
$\Omega_{LK}$ by Theorem 3.1.  $G$-equivariance of the connection is
easy to see.

\end{proof}

\textbf{Remark} In the connection $\Omega _{LK}$  if make $\mu _s
=\kappa $ for all $s$ and $m_i =m$ for all $i$ then we obtain
Marin's connection. Above theorem produces flat, $G$-equivariant
connections with more parameters. It also explains the number
$\alpha _{i,j}$ in Marin's construction.

\begin{defi}[Generalized LK representations for general complex braid groups]
Following notations introduced above. Since $\Omega _{LK}$ is
$G$-invariant, it induces a flat connection $\bar{\Omega } _{LK}$ on
the quotient bundle $M_G \times _{G} V_G $.  The generalized
Lawrence-Krammer representation of the braid group $A_G $ is defined
as the monodromy representation of $\bar{\Omega } _{LK} $.

\end{defi}

 Suppose $P_1 ,\ P_2 ,\ \cdots , \
P_{N_G } $ are all equivalent classes of $P$ under the equivalence
relation $'\sim '$. The following lemma is essentially from
$\cite{Ma2}$.

\begin{lem}
For any $1\leq N\leq N_G $, the subspace $V_N =\oplus _{i\in P_N }
\mathbb{C} v_i  \subset V_G $ is a subrepresentation.

\end{lem}

\begin{pf}
We only need to observe that $w(i) \sim i $ for any $i\in P $ and
$w\in G$, and $p_i (v_j ) = 0 $ if $i \nsim j $.
\end{pf}

\section{Basic Properties about  $B_G (\Upsilon)$}

Suppose $G\subset U(V)$ is a finite pseudo reflection group. Define
$B_G (\Upsilon)$ as in Definition 1.1. When $G$ is a complex
reflection group, there is a bijection from $R$ to $\mathcal {A} $:
$s \mapsto H_s $. So we can use $R$ as the indices set $P$ of
reflection hyperplanes. In these cases, for $s_1 ,s_2 \in R $,
$R(s_1 ,s_2 )= \{ s\in R | s(H_{s_2}) =H_{s_1 } \} =\{s\in R | ss_2
s=s_1 \} $.

\begin{thm}

 When $G$ is finite then $B _G
(\Upsilon )$  is a finite dimensional algebra. Moreover, the map $w
\mapsto T_w $ for $w\in G $ induce an injection  $j: \mathbb{C} G
\rightarrow B _G(\Upsilon )$.

\end{thm}

\begin{proof}
First by using relation (3), we can identify any word made from the
set $\{w\in G \} \coprod \{e_i \} _{i\in P } $ with a word of the
form $T_w e_{i_1} e_{i_2 }\cdots e_{i_k }$ where $w\in G $. We
define the e-length of such a word as $k$.  In this word if two
neighboring $e_{i_v }$ and $e_{i_{v+1 }}$ don't commute with each
other, then for $e_{i_v} e_{i_{v+1}}$, condition in $(5)$ of
Definition 1.1 is satisfied as shown by the next lemma.

\begin{lem}
If two pseudo reflection $s_1 $ and $s_2 $ don't commute with each
other, suppose the reflection hyperplane of $s_1 (s_2 )$ is $H_{i_1
} (H_{i_2 } ) $, then $\{i_1 ,i_2  \} \subsetneq \{k\in P | H_k
\supseteq H_{i_1 } \cap H_{i_2 } \} $.
\end{lem}

\begin{proof} We suppose $\{i_1 ,i_2  \} = \{k\in P | H_k \supseteq H_{i_1 }
\cap H_{i_2 } \} $. Let $L= H_{i_1 }\cap H_{i_2 }$, and $<\ ,\  >$
being a $G$-invariant inner product on $V$. Chose $v_k \in H_{i_k }$
such that $v_k \perp L $ according to $<\ ,\ >$ for $k=1,2$. Suppose
$\{ v_3 ,\cdots , v_N \}$ is a basis of $L$, then $\{ v_1 ,v_2
,\cdots ,v_N \}$ is a basis of $V$. Now since $s_1 (H_{i_2 })$ is
another reflection hyperplane containing $L$ and $s_1 (H_{i_2 })\neq
H_{i_1 }$, so we have $s_1 (H_{i_2 }) =H_{i_2 }$, which implies
$s_1$ can be presented as a diagonal matrix according to the basis
$\{v_1 ,\cdots ,v_N \}$. Similarly $s_2 $ can be presented by a
diagonal matrix according to the same basis. So $s_1 s_2 =s_2 s_1 $
which is a contradiction.
\end{proof}

The first statement of theorem 5.1 follows from the next lemma.

\begin{lem}
The algebra $B _G (\Upsilon )$ is spanned by the set
$$ \{T_w e_{i_1 }\cdots e_{i_ M}
| w\in G ; \ e_{i_u } e_{i_v } =e_{i_v } e_{i_u } and \ i_v \neq i_u
\ if\ u\neq v \ ; M\geq 0 \}$$

\end{lem}

\begin{proof} Let $A $ be the subspace in $B _G (\Upsilon )$ spanned
by elements listed in the lemma.  We only need to prove any word
$T_w e_{i_1} e_{i_2 }\cdots e_{i_k }$ represents an element in $A$.
We do it by induction on e-length of such words. First this is true
if $K=1$. Suppose it is true for $K\leq M $. Now consider a word
$x=we_{i_1 }\cdots e_{i_{M+1} }$. If there are two neighboring
$e_{i_v }$, $e_{i_{v+1 }}$ don't commute, then Lemma 5.1 enable us
to apply (5) or (6) in Definition 1.1 to identify $x$ with a linear
sum of words whose e-tail length are smaller than $M+1$. Suppose all
$e_{i_v }$'s in $x$ commute with each other, if there are $v_1 ,v_2
$ such that $i_{v_1 } =i_{v_2 } $, we use permutations between
$e_{i_v }$'s to identify $x$ with a word $y=we_{j_1 }\cdots
e_{j_{M+1 } }$ such that $j_1 =j_2 $. So $x=y=m_{j_1 }w e_{j_2
}\cdots e_{j_{M+1 } } $ by relation (2) of Definition 1.1. If all
$e_{i_v }$'s commute and all $i_v $'s are different then $x\in A $,
and induction is completed. For the second statement of theorem 5.1,
it isn't hard to see the following map
$$T_w \mapsto w,\ for\ w\in G;\ e_i \mapsto\ 0, \ for \ i\in P$$
extends to a surjection $\pi : B _G (\Upsilon ) \rightarrow
\mathbb{C} G $, and $\pi \circ j =id $. So $j$ is injective.
\end{proof}

This completes the proof of Theorem 5.1.
\end{proof}

By Theorem 5.1, $\mathbb{C} G$ is naturally embedded in $B_G
(\Upsilon )$. For saving notations from now on we always think
$\mathbb{C} G$ to be included in $B_G (\Upsilon )$, and denote $T_w
$ simply as $w$. The next lemma reduce one parameter in $B _G
(\Upsilon )$.

\begin{lem}
For $\lambda \in \mathbb{C} ^{\times }$, Let $\mu ^{\prime } _s
=\lambda \mu _s $ for $s \in R$, and Let $m^{\prime } _i =\lambda
m_i $ for $i\in P$. Let $\Upsilon ^{\prime }=\{\mu ^{\prime } _s ,
m^{\prime } _i \}_{s \in R ,i\in P}$, then $B _G (\Upsilon ^{\prime
})\cong B _G (\Upsilon )$.

\end{lem}

\begin{proof}
Denote the generators of  $B _G (\Upsilon ^{\prime })$ appeared in
Definition 1.1 as $S_i $'s and $E_i $'s . Then
$$s_i \mapsto S_i , \ e_i \mapsto \lambda E_i \ for \  i\in P $$
extend to an isomorphism from $B _G (\Upsilon )$ to $B _G (\Upsilon
^{'} )$.

\end{proof}

The following lemma can be found in \cite{Ma2}.
\begin{lem}[Marin]
For two different hyperplane $H_i ,H_j \in \mathcal {A}$, If $s\in R
$ satisfies $s(H_j ) =H_i $, then $s$ fix all points in $H_i \cap
H_j $. So, $R(i,j) =\{ s\in R | s(H_j )=H_i ; H_s \supseteq H_i \cap
H_j \}.$

\end{lem}

\begin{proof}
Let $ <\ ,\ > $ be a $G$-invariant inner product on $V$.  Let
$\epsilon $ be the exceptional eigenvalue of $s$, and let $u $ be an
eigenvalue of $s$ with eigenvalue $\epsilon$. Let $u _i $, $ u _j $
be some nonzero vectors perpendicular to $H_i $, $H_j $
respectively. Then $u \perp H_s $. The action of $s$ on $V$ can be
written as $s(v) = v- (1-\epsilon ) \frac {<v, u
>}{<u ,u >} u $.  Now $s(H_j )=H_i $ implies
$s(u _j ) = u _j - (1-\epsilon ) \frac {<u _j , u
>}{<u ,u >} u  =\lambda u _i $ for some $\lambda
\neq 0$. Denote $(1-\epsilon ) \frac {<u _j , u
>}{<u ,u >}$ as $\kappa $. The condition that  $H_i $ is
different from $H_j $ implies $\kappa \neq 0$. So we have $u = \frac
{1}{\kappa} (u _j - \lambda u _i ) \perp H_i \cap H_j $, and $s$ fix
all points in $H_i \cap H_j $.

\end{proof}

There exists a natural anti-involution  on $B_G (\Upsilon )$ which
may be used to construct a cellular structure as follows.

\begin{lem}

The following map extends to an anti-involution $*$ of $B_G
(\Upsilon )$  $$ w\mapsto w^{-1}\   for\  w\in G \subset B_G
(\Upsilon ),\  e_i \mapsto e_i \ for\  all\  i\in P $$ if $\mu _s =
\mu _{s^{-1 }}$ for any $s\in R $.

\end{lem}

\begin{proof}
We only need to certify $*$ keep all relations in Definition 1.1. As
an example for relation (5), on the one hand $* (e_i e_j ) = e_j e_i
$, on the other hand  $* [ (\sum _{s\in R(i,j)} \mu _s s ) e_j ] =
e_j (\sum _{s\in R(i,j)} \mu _s s^{-1 } ) = (\sum _{s\in R(i,j)} \mu
_s s^{-1}) e_i = (\sum _{s\in R(j,i) } \mu _{s^{-1} } s ) e_i =
(\sum _{s\in R(j,i) } \mu _{s } s ) e_i .$
\end{proof}

\nd{\bf Flat Connections   }

Define a formal connection $ \bar{ \Omega} _G = \kappa \sum _{i\in P
} (\sum _{s:i(s)=i }\mu _s s -e_i ) \omega _i .$

Suppose $\rho : B _G (\Upsilon ) \rightarrow End(E)$ is a finite
dimensional representation. On the vector bundle $M_G \times E $, we
define a connection $ \rho (\bar{\Omega } _G ) = \kappa \sum _{i\in
P } (\sum _{s:i(s)=i }\mu _s \rho (s) -\rho (e_i )) \omega _i$ where
$\kappa \in \mathbb{C}^{\times}$. Let $G$ acts on $M_G \times E $ as
$w\cdot (x,v)= (wx, \rho (w) v )$ for $w\in G$ and $(x,v)\in M_G
\times E$.

\begin{prop}
The connection $\rho (\bar{\Omega } _G ) $  and  $ \bar{ \Omega} _G
$ are flat and $G$-invariant.

\end{prop}

\begin{proof}
 It is enough to deal with the case $\kappa =1$. By Proposition 2.2,
  to show the $G$-invariance we only need to prove
\begin{equation}
\sum _{s:i(s)=i }\mu _s \rho (w) \rho (s) \rho (w)^{-1} - \rho
(w)\rho (e_i ) \rho (w)^{-1} = \sum _{s:i(s)=w(i) }\mu _s \rho (s)
-\rho (e_{w(i)} )
\end{equation}

for any $w\in G$.  By (3) of Definition 1.1, we have $\rho (w)\rho
(e_i ) \rho (w)^{-1}=\rho (we_i w^{-1})=\rho (e_{w(i)})$. We also
have $\{wsw^{-1} | i(s) = i \}= \{s| i(s)= w(i) \}$  and $\mu _s =
\mu _{wsw^{-1}}$, so identity (12) follows.

Let $L$ be any codimension 2 edge for the arrangement $\mathcal
{A}$, and let $H_{i_1 }, \cdots , H_{i_N } $ be all the hyperplanes
in $\mathcal {A}$ containing $L$. By Theorem 3.1, to prove $\rho
(\bar{\Omega } _G ) $ is flat we need to show for any $u$
\begin{equation}
[ \sum _{s: i(s)=i_u } \mu _s \rho (s) -\rho ( e_{i_u } ) , \sum
_{v=1} ^{N} ( \sum _{s:i(s) =i_v } \mu _s \rho (s) -\rho (e_{i_v }
)) ] =0.
\end{equation}

Now remember the connection $ \kappa \sum _{i\in P } (\sum
_{s:i(s)=i }\mu _s \rho (s)  \omega _i )$ is flat by proposition
2.1, so (13) is equivalent to

\begin{equation}
[ \sum _{s: i(s)=i_u } \mu _s \rho (s) , \sum _{v=1} ^{N} \rho
(e_{i_v } ) ]+[ \rho ( e_{i_u } ) , \sum _{v=1} ^{N} ( \sum _{s:i(s)
=i_v } \mu _s \rho (s) -\rho (e_{i_v } )) ]=0
\end{equation}

Because for any $s$ such that $i(s)= i_u $, there is $\{ s(H_{i_1
}), \cdots ,s(H_{i_N })\}= \{H_{i_1 }, \cdots ,H_{i_N } \}$, so
\begin{equation}
\begin{split}
\rho (s) \sum_{v=1 } ^N \rho (e_{iv }) - \sum_{v=1 } ^N \rho (e_{iv
})\rho (s) &=(\rho (s) \sum_{v=1 } ^N \rho (e_{iv }) \rho (s)^{-1} -
\sum_{v=1 } ^N \rho (e_{iv }) )\rho (s)\\
& = (\sum_{v=1 }^N \rho (e_{s(i_v) }) - \sum _{i=1} ^N \rho (e_{i_v
}))=0
\end{split}
\end{equation}
So (14) is equivalent to
\begin{equation}
[ \rho ( e_{i_u } ) , \sum _{v=1} ^{N} ( \sum _{s:i(s) =i_v } \mu _s
\rho (s) -\rho (e_{i_v } )) ]=0.
\end{equation}

We define $I_1 =\{ 1\leq v\leq N | s(i_v) =i_u , for\ some\ s\in
\mathcal {R}\}$, and $I_2 = \{ 1\leq v\leq N | s(i_v) \neq i_u ,
for\ any\ s\in \mathcal {R}\}$. There is $\{1,2,\cdots, N \}=I_1
\coprod I_2$.

\begin{equation}
\begin{split}
[ \rho ( e_{i_u } ) , \sum _{v=1} ^{N} ( \sum _{s:i(s) =i_v } \mu _s
\rho (s) -\rho (e_{i_v } )) ] &= -\sum _{v\in I_1  } (\rho(e_{i_u }
e_{i_v }) -\sum _{s: s(i_v) =i_u } \mu _s \rho(e_{i_u }) \rho (s) )\\
&+\sum_{v\in I_1 } (\rho(e_{i_v } e_{i_u }) -\sum _{s: s(i_u ) =i_v
} \mu _s \rho (s)
\rho(e_{i_u }) )\\
&+ \sum _{v\in I_2 } (\rho(e_{i_u } e_{i_v}) -\rho(e_{i_v } e_{i_u })) \\
&=0.
\end{split}
\end{equation}

Where we used relation (5),(6) in Definition 1.1.

Flatness of $ \bar{ \Omega} _G $ can be proved similarly.

\end{proof}

\begin{thm}
Using notations in section 4. The map $w\mapsto \iota (w) $, $e_i
\mapsto p_i $ extends to a representation $B_G (\Upsilon)
\rightarrow End(V_G )$. So from $B_G (\Upsilon )$ we can obtain the
generalized Lawrence-Krammer representation.
\end{thm}
\begin{proof}
We only need to certify that $\iota (w)$'s and $p_i $'s satisfy
those relations in Definition 1.1. Relation $(0)$ is evident.
Relation $(1)$ and $(1)^{'}$ are because $p_i $ is a projector to
$\mathbb{C}v_i$. Relation $(3)$ is by definition of $p_i $ and the
fact $\alpha _{i,j}= \alpha _{w(i),w(j)}$. When $R(i,j)=\emptyset$,
by definition we have $p_i p_j =p_j p_i =0$ so relation $(4), (6)$
follows. For any $i,j,k$, we have

$p_i p_j (v_k )= \alpha _{j,k} p_i (v_j) = \alpha _{j,k} \alpha
_{i,j} v_i $ and

$(\sum _{s\in R(i,j)} \mu _s \iota (s)  )p_j (v_k )= \alpha _{j,k}
(\sum _{s\in R(i,j)} \mu _s \iota (s)  )(v_j)= \alpha _{j,k} (\sum
_{s\in R(i,j)} \mu _s ) v_i  =\alpha _{j,k} \alpha _{i,j} v_i , $

so relation $(5)$ is certified.
\end{proof}

Suppose $\Gamma$ is a finite type simply laced Dynkin diagram,
denote the associated Coxeter group and Artin group as $W_{\Gamma}$,
$A_{\Gamma}$ respectively. Suppose $W_{\Gamma}$ is realized as a
reflection group in $U(V)$. In this case the data $\Upsilon$
consists of two constants $m, \mu$ since all reflections of
$W_{\Gamma}$ lie in the same conjugacy class. Suppose $\mu\neq 0$,
so by Lemma 5.3 we can set $\mu =1$. Thus we denote the algebra of
Definition 1.1 for $W_{\Gamma}$ as $B_{W_{\Gamma}} (m)$. Set
$M_{\Gamma} = V\setminus \cup_{i\in P} H_i $. Let $(E,\rho)$ be a
finite dimensional representation of $B_{W_{\Gamma}} (m )$. By
proposition 5.1, the connection
$$ \rho (\bar{\Omega} _{\Gamma})= \kappa \sum _{i\in P}
( \rho (s_i ) -\rho (e_i ) ) \omega_i $$ induces a flat connection
on the bundle $M_{\Gamma} \times _{W_{\Gamma}} E $, which is a
linear bundle on $M_{\Gamma} /W_{\Gamma}$. Denote the resulted
monodromy representation of $\pi _1 ( M_{\Gamma} / W_{\Gamma} )
\cong A_{\Gamma} $ as $\bar{\rho}$. We have

\begin{thm}
If $m\notin \mathbb{Z}$ or $m\in \mathbb{Z}$ with $m>3$,  the
monodromy representations $\bar{\rho}$ of $A_{\Gamma}$ constructed
above factor through the simply laced BMW algebra $B_{\Gamma}(\tau ,
l)$, for $\tau = \frac{q^{1-m} - q^{m-1} + q^{-1} -q}{q^{-1} -q}$,
$l=q^{m-1} $. Where $q= \exp \kappa \pi \sqrt{-1}$.
\end{thm}

\begin{proof}
Suppose $\Sigma$ is the set of nodes of $\Gamma$, $\{ \sigma_i
\}_{i\in \Sigma}$ is the set of generators of $A_{\Gamma}$ in a
canonical presentation. For $i\in \Sigma$, set $$X_i =\bar{ \rho }
(\sigma_i ) ,\  E_i = \frac {q^{-1} -q}{l} (\bar{\rho } (\sigma_i )
^2 + (q^{-1} -q) \bar{\rho } (\sigma_i )-1).$$ We need to show $\{
X_i , E_i \}_{i\in \Sigma}$ satisfies relations of $B_{\Gamma} (\tau
,l)$ in table 2. The proof is completely similar to Theorem 3.2, so
we content with giving  a sketch. Denote the number of nodes in
$\Gamma$ as $n(\gamma )$. We only consider the cases when $\Gamma$
is irreducible. When $n(\Gamma) =1$ or $2$, the Artin group
$A_{\Gamma}$ is braid group $B_2$, $B_3$ respectively. So the
statement follows from Theorem 3.2.  Suppose $n(\Gamma) \geq 3$. The
fact that $\Gamma$ is simply laced enable us to reduce the statement
of these cases to cases when $n(\Gamma) =1$ or $2$, by using Theorem
2.1.  Suppose  $i,j \in \Sigma $ and $i\sim j$.  Then the parabolic
subgroup $W_{i,j}$ of $W_{\Gamma}$ generated by $s_i ,s_j$ is
isomorphic to the symmetric group $S_3$. By applying Theorem 2.1 to
$W_{i,j}$, we prove that $X_i ,X_j , E_i ,E_j$ satisfies relations
in table 2.  The cases when $i \nsim j$ can be proved similarly.

\end{proof}

\section{Cases of Dihedral Groups}

\nd{\bf  Dimension and Basis} Denote the dihedral group of type $I_2
(m)$ as $G_m $. The arrangement of its reflection hyperplanes can be
explained with the following Figure 2.
\begin{figure}[htbp]

  \centering
  \includegraphics[height=4.5cm]{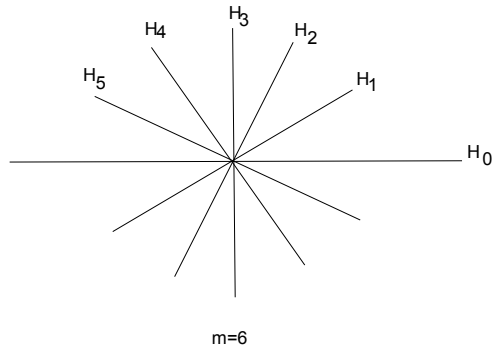}
  \caption{The arrangement of $G_6$ }
\end{figure}

There are $m$ lines(hyperplanes) passing the origin. The angle
between every two neighboring lines is $\pi /m$. Suppose the
$x$-axis is one of the reflection lines and denote it as $H_0 $, we
denote these lines by $H_0 , H_1 ,\cdots , H_{m-1} $ in
anticlockwise order as shown in above graph.  Denote the reflection
by $H_i $ as $s_i $. The set of reflections in $G_m$ is $R=\{ s_i \}
_{0\leq i\leq m-1 }$. Denote the rotation of $2j \pi /m $ in
anticlockwise order as $r_j $. It is well known that $G_m$ is
generated by $s_0 ,s_1 $ with the following presentation
$$< s_0\ ,s_1 , \ | ( s_0 s_1 )^m =1 , s_0 ^2 =s_1 ^2 =1 >.$$

As a set $G_m =\{ s_i ,r_i \}_{0\leq i\leq m-1}$.  Under this
presentation, $s_i $ can be determined inductively in the following
way.  $s_1 =s_1 $, $s_2 = s_1 s_0 s_1 $ and $s_i = s_{i-1} s_{i-2}
s_{i-1} $. By $[s_i s_j \cdots ]_{k}$ we denote the length $k$ word
starting with $s_i s_j $, in which $s_i $ and $s_j $ appear
alternatively. The word $[\cdots s_0 s_1]_k$ is defined similarly.
Then $s_i = [s_1 s_0 \cdots ]_{2i-1} $, where $s_m =[s_1 s_0 \cdots
]_{2m-1 } =s_0 $.

As for the algebra $B_{G_m} (\Upsilon )$ , we choose the index set
$P_{G_m }= \{0,1,2,\cdots , m-1 \}$. $B_{G_m } (\Upsilon )$ is
generated by $\{ s_i , e_i \} _{i\in P_{G_m} }$. The data $\Upsilon
$ is $\{ \mu _i , \tau _i \}_{i\in P_{G_m}}$.  For later
convenience, for $k \in \mathbb{Z}$, we define $s_k $, $H_k$, $e_k
$, $\mu _k $, $\tau _k $ as $s_{[k]}$, $H_{[k]}$, $e_{[k]}$, $\mu
_{[k]}$, $\tau _{[k]}$ respectively. Where  $[k]$ is the unique
number in $\{0,1,\cdots , m-1 \}$ such that $k\equiv [k] mod (m
\mathbb{Z})$.

The structure of $B_{G_m}(\Upsilon )$ when $m$ is odd is quite
different from cases when $m$ is even.

\begin{itemize}
  \item When $m=2k+1$ is odd, we have $i\sim j$ for any $i,j \in P_{G_{m}}$.
  Which implies $\mu _i =\mu _j $ and $\tau _i =\tau _j$ for any $i,j\in P_{G_{m}}
  $. So the data $\Upsilon$ consists of $\{ \tau _0 , \mu _0 \}$
  essentially. We always suppose $\mu_0 \neq 0$, and by Lemma 5.3 we
  can suppose $\mu_0 =1$. When $i+j\in 2\mathbb{Z}$, $R(i,j)= \{ s_{\frac{i+j}{2}} \}$; when
$i+j+1 \in
  2\mathbb{Z}$, $R(i,j) =\{ s_{\frac{i+j+m}{2}} \}$.
  \item When $m=2k$ is even, $i\sim j$ if and only if $i+j \in
  2\mathbb{Z}$. Which implies  $\mu _i =\mu _j $ and $\tau _i =\tau
  _j$  if $i+j \in 2\mathbb{Z}$. So the data $\Upsilon$ consists of $\{ \mu _0 , \mu _1 , \tau _0 ,\tau _1
  \}$ essentially. When  $i+j \in 2\mathbb{Z}$,  $R(i,j)=
  \{s_{\frac{i+j}{2}}, s_{\frac{i+j+m}{2}}
  \}$; when $i+j+1 \in 2\mathbb{Z}$, $R(i,j)= \emptyset$.
\end{itemize}

It is easy to see when $m=2k$ is even, the relation $(1)^{'}$ in
Definition 5.1 for $B_{G_m}(\Upsilon )$ is equivalent to: $ s_{i+k}
e_i =e_i s_{i+k} =e_i $ for any $i$.

\begin{thm}
$(1)$ When $m$ is odd, the algebra $B_{G_m}(\Upsilon)$ has dimension
$2m + m^2 $, and has the set $\{ s_i , r_i \} _{0\leq i\leq m-1 }
\cup \{ s_i e_j \} _{0\leq i,j\leq m-1}$ as a basis.

$(2)$ When $m=2k$ is even, the algebra $B_{G_m}(\Upsilon)$ has
dimension $2m+ \frac{m^2}{2}$,and has the set $\{ s_i , r_i \}
_{0\leq i\leq m-1 } \cup \{ s_i e_{2j} \}_{0\leq i,j \leq k-1} \cup
\{ s_i e_{2j+1} \}_{0\leq i,j\leq k-1 }$ as a basis.

\end{thm}

When $m$ is odd, let $A_m $ be the vector space spanned by some
generators $\{S_i ,R_i \}_{0\leq i\leq m-1} \cup \{ T_{i,j}\}_{0\leq
i,j \leq m-1}$. For convenience for any $i,j \in \mathbb{Z}$, let
$S_i = S_{[i]}$, $R_{i}= R_{[i]}$ and $T_{i,j} = T_{[i],[j]}$.
Define a product on $A_m $  by the following relations $(1), (2),
(3)$.

\begin{enumerate}
\item[(1)] $S_i S_j = R_{i-j}$,$R_i R_j = R_{i+j}$, $S_i R_j =
S_{i-j}$, $R_j S_i = S_{i+j}$.
\item[(2)] $S_l T_{i,j}= T_{l-i+j,
j}$,$ T_{i,j} S_l =T_{i-j+l, 2l-j}$, $R_l T_{i,j} =T_{i+l ,j }$,
$T_{i,j} R_l =T_{i-l, j-2l}$.
\item[(3)] \[
 T_{i,j} T_{p,q}= \left\{
\begin{array}{lll}
 \tau _q T_{i-p+q, q} & \ when\  2p-j-q \in m\mathbb{Z};\\
 T_{v  _{i,j,p,q},q} & \ when\   2p-j -q  \notin m\mathbb{Z},  and\   [2p-j]+q\in 2\mathbb{Z};\\
 T_{u _{i,j,p,q},q} & \ when\   2p-j -q \notin m\mathbb{Z}, and\
[2p-j]+q \notin  2\mathbb{Z} .
\end{array}
\right.
\]

\end{enumerate}

Where $v _{i,j,p,q}= (2i+[2p-j]+q-2p)/2$, $u _{i,j,p,q}=
(2i+[2p-j]+q+m-2p)/2.$

\begin{lem}
Above product makes $A_m$ into an associative algebra.
\end{lem}

\begin{proof}
In fact above identities are obtained by "looking $(S_i ,  R_i ,
T_{i,j})$ as $(s_i , r_i , s_i e_j )$". We have an indirect proof as
follows. Denote the $m$-dimensional representation of $B_G
(\Upsilon)$ defined in Theorem 5.2 as $\rho_{LK}$, the irreducible
representations of $B_G (\Upsilon)$ induced by the surjection $\pi :
B_G (\Upsilon)\rightarrow \mathbb{C}G$ as $\rho_1 ,\cdots ,\rho_l $.
Denote the parameter space of all $\Upsilon$'s as $\Lambda$. By
similar argument with the proof of $(2)$ of Proposition 7.1, we can
show there is a dense open subset $\mathcal {D}$ of $\Lambda$ such
that if $\Upsilon \in \mathcal {D}$ then the related representation
$\rho_{LK}$ is irreducible. In there cases by Wedderburn-Artin
Theorem we have $\dim B_{G}(\Upsilon) \geq \sum_{i} ( \dim \rho_i
)^2 + (\dim \rho_{LK})^2 = 2m+m^2$. Since the set $\{ s_i , r_i \}
_{0\leq i\leq m-1 } \cup \{ s_i e_j \} _{0\leq i,j\leq m-1}$ always
spans $B_G (\Upsilon)$, so we know when $\Upsilon \in \mathcal {D}$,
this set is a basis of $B_G (\Upsilon)$. Thus the product of $A_m$
is associative if $\Upsilon \in \mathcal {D}$.  So the produce is
associative for all $\Upsilon$.
\end{proof}
\nd{\bf Proof of theorem 6.1}
 When $m$ is odd, denote the algebra above as $A_m (\Upsilon)$. we define a map  $\phi $ as: $\phi (S_i )=s_i $, $\phi (R_i )=r_i
 $, for $0\leq i\leq m-1$; $\phi (T_{i,j}) = s_i e_j $ for $0\leq i,j\leq m-1
 $. It is easy to see $\phi$ extends to a morphism $\phi$ from $A_m (\Upsilon) $
 to  $B_{G_m }(\Upsilon)$. Inversely the map $\psi$: $\psi (s_i) =S_i ,\  \psi (e_i )= T_{i,i}
 \ for\ 0\leq i\leq m-1 $ extends to a morphism $\psi$ from  $B_{G_m
 }(\Upsilon)$ to  $A_m (\Upsilon)$. Since $\psi \phi =id$ and $\phi \psi =id$,
 we know $B_{G_m }(\Upsilon)$ is isomorphic to $A_m (\Upsilon)$ and statement
 $(1)$ follows. The statement $(2)$ can be proved similarly by
 constructing an actual algebra with dimension $2m + \frac{m^2}{2}$
 and prove it is isomorphic to $B_{G_m}(\Upsilon)$.\\

\nd{\bf  Cellular Structures When $m$ is Odd  }   Suppose $m=2k+1$.
let $(\Lambda , M , C ,* )$ be the cellular structure of $\mathbb{C}
G_m $. The algebra $B_{G_m }(\Upsilon)$ has a cellular structure
$(\bar {\Lambda },\bar { M}, \bar {C} ,*)$ as follows.

\begin{itemize}
  \item $\bar {\Lambda } = \Lambda \coprod \{\lambda _{LK } \}$.We
keep the original partial order in $ \Lambda$£» and for any $\lambda
\in \Lambda $, let $\lambda _{LK} \prec \lambda $.

  \item For $\lambda \in \Lambda $, set $ \bar {M} (\lambda )=
M(\lambda ) $ and  $\bar {M} (\lambda _{LK} ) = \{0,1,\cdots ,m-1
\}$.

  \item For $\lambda \in \Lambda $, and $S,T \in \bar {M}
(\lambda)$ set  $\bar {C} _{S,T} ^{\lambda }= C_{S,T} ^{\lambda }$.
For $i,j \in \bar {M} (\lambda _{LK} ) $, set  $\bar {C} _{i,j}
^{\lambda _{LK }} = s_{\frac{i+j}{2}} e_j $ if $i+j \in
2\mathbb{Z}$, and $\bar {C} _{i,j} ^{\lambda _{LK }} =
s_{\frac{i+j+m}{2}} e_j $ if $i+j \notin 2\mathbb{Z}$.

  \item Define $*$ to be the involution  in Lemma 5.5.
\end{itemize}

\begin{thm}
Above $(\bar {\Lambda },\bar { M}, \bar {C} ,*)$ defines a cellular
structure for $B_{G_m }(\Upsilon)$.
\end{thm}
\begin{pf}
Recall the definition of cellular algebras in section 2. $(C1)$
follows by Theorem 6.1. $(C2)$ is because $*( s_{\frac{i+j}{2}} e_j)
= e_j s_{\frac{i+j}{2}}  =  s_{\frac{i+j}{2}} e_i $. $(C3)$ is by
the following computation. $s_l (\bar {C} _{i,j} ^{\lambda _{LK
}})=\bar {C} _{2l-i,j} ^{\lambda _{LK }} $ for any $l,i,j$;  $e_l
(\bar {C} _{i,j} ^{\lambda _{LK }})= \bar {C} _{l,j} ^{\lambda _{LK
}} $ if $l\neq i$ ; $e_i (\bar {C} _{i,j} ^{\lambda _{LK }})=\tau _0
\bar {C} _{i,j} ^{\lambda _{LK }}. $
\end{pf}

\begin{rem}
From above proof we see the cellular representation corresponding to
$\lambda _{LK }$ is the infinitesimal LK representation of Marin.
\end{rem}

\nd{\bf  Cellular Structures When $m$ is Even } Suppose $m=2k$.
Still denote the cellular structure of $\mathbb{C} G_m $ as
$(\Lambda , M , C ,* )$. The algebra $B_{G_m }(\Upsilon)$ has a
cellular structure $(\bar {\Lambda },\bar { M}, \bar {C} ,*)$ as
follows.

\begin{itemize}
  \item $\bar{\Lambda}=\Lambda \cup \{\lambda _{LK ^0 } , \lambda
_{LK ^1 } \}$. We keep the partial order in $\Lambda$, and let $LK^i
\prec \lambda $  for any $\lambda \in \Lambda$ and any $i$.
  \item For $\lambda \in \Lambda $, $\bar{M}(\lambda )
=M(\lambda)$.  $\bar{M}(\lambda _{LK ^0 })=\{0,2,\cdots ,2k-2
\}$.$\bar{M}(\lambda _{LK ^1 })=\{1,3,\cdots ,2k-1 \}$.

  \item  For $\lambda \in \Lambda $, $S,T
\in \bar{M} (\lambda )$, let $\bar{C} ^{\lambda } _{S,T} =C^{\lambda
} _{S,T}$.

  \item $\bar {C} ^{\lambda _{LK ^0 }} _{2i,2j } =\frac{1}{2}(s_{i+j}
  +s_{i+j+k})e_{2j}$, $\bar {C} ^{\lambda _{LK ^1 }} _{2i+1,2j+1 } =\frac{1}{2}(s_{i+j+1}
  +s_{i+j+k+1})e_{2j+1}$.

  \item Let $*$ be the involution in lemma 5.5.

\end{itemize}

\begin{thm}
Above $(\bar {\Lambda },\bar { M}, \bar {C} ,*)$ defines a cellular
structure for  $B_{G_m }(\Upsilon)$ .
\end{thm}

\begin{proof}
$(C1)$ follows from Theorem 6.1. $(C2)$ is certified similarly.
$(C3)$ follows from the following computation. $s_l (\bar {C}
^{\lambda _{LK ^0 }} _{2i,2j })= \bar {C} ^{\lambda _{LK ^0 }}
_{2(l-i),2j }$; $e_l (\bar {C} ^{\lambda _{LK ^0 }} _{2i,2j })=0$ if
$l$ is odd; $e_{2i} (\bar {C} ^{\lambda _{LK ^0 }} _{2i,2j })= \tau
_0 \bar {C} ^{\lambda _{LK ^0 }} _{2i,2j }$; $e_l (\bar {C}
^{\lambda _{LK ^0 }} _{2i,2j })= (\mu _{p+i} +\mu _{p+i+k}) \bar {C}
^{\lambda _{LK ^0 }} _{l,2j }$ if $l=2p$ is even and $l\neq 2i $.

$s_l (\bar {C} ^{\lambda _{LK ^1 }} _{2i+1,2j+1 })= \bar {C}
^{\lambda _{LK ^1 }} _{2(l-i-1)+1,2j+1 }$; $e_l (\bar {C} ^{\lambda
_{LK ^1 }} _{2i+1,2j+1 })=0$ if $l$ is even; $e_{2i+1} (\bar {C}
^{\lambda _{LK ^1 }} _{2i+1,2j+1 })= \tau _1 \bar {C} ^{\lambda _{LK
^1 }} _{2i+1,2j+1 };$  $e_l (\bar {C} ^{\lambda _{LK ^1 }}
_{2i+1,2j+1 })= (\mu _{p+i+1} +\mu _{p+i+1+k}) \bar {C} ^{\lambda
_{LK ^1 }} _{2p+1,2j+1 }$ if $l=2p+1$ is odd and $l\neq 2i+1$.

\end{proof}

\begin{rem}
The two representations corresponding to $\lambda _{LK ^0 }$,
$\lambda _{LK ^1 }$ are components of the infinitesimal LK
representations of Marin as in Theorem5.2.
\end{rem}

\section{ $H_3$ Type}
 The Coxeter group $G_{H_3}$ of type $H_3$ is the symmetric group of a
regular dodecahedron(or a regular icosahedron). It is generated by
$s_0$, $s_1$, $s_2$ with relations:
\begin{align*}
&a)\  s_i ^2 =1\  for\  all\  i's.        &b)&\  s_0 s_1 s_0 s_1 s_1
= s_1 s_0
s_1 s_0 s_1.\\
&c)\  s_0 s_2 =s_2 s_0 .       &d)&\  s_1 s_2 s_1 =s_2 s_1 s_2.
\end{align*}

We have: $|G_{H_3}|= 120$, $|R|= 15$.

The group $G_{H_3}$ has a nontrival center element $c=(s_2 s_0 s_1
)^5 $ which is also the longest element. Denote the other $12$
reflections arbitrarily as $s_3 ,\ \cdots ,\ s_{14}$ so $R= \{s_i
\}_{0\leq i\leq 14}$. Denote the reflection hyperplane of $s_i$ as
$H_i$, naturally set the index set of reflection hyperplanes as
$P=\{ 0,1,\cdots ,14 \}$.

In the following Figure 3, the dotted lines show the intersection of
three reflection hyperplanes with the front surface of the
dodecahedron. For $s,\ s^{'} \in R$, we say $s$ is perpendicular to
$s^{'}$ and denote $s \bot s^{'}$ if $ss^{'} =s^{'}s$. From Figure 3
we see directly that $'\bot'$ is a equivalent relation in $R$ ( the
proof of this fact is only simple but lengthy computations ).
According to it $R$ is decomposed into 5 equivalent classes
$\mathscr{R} =\{ R_1 , \cdots ,R_5 \}$, each class consists of 3
elements. Let $w_0 = s_1 s_2 s_1 s_0 s_1 s_0 s_1 $.  A typical
equivalent class is $\{ s_0 , s_2 , s_{i_0 } \}$ where $s_{i_0}= w_0
^{-1} s_0 w_0 $. Any way we suppose $R_{\alpha} =\{ i_{\alpha}
,j_{\alpha} ,k_{\alpha } \}$ for $1\leq \alpha \leq 5$ and let $R_1
=\{ s_0,s_2, s_{i_0} \}$. The conjugating action of $G_{H_3}$ on $R$
induces an action of the same group on $\mathscr{R}$, because $s_i
\perp s_j $ implies $ws_i w^{-1} \perp ws_j w^{-1} $. Since
$G_{H_3}$ acts on $R$ transitively by conjugation, the action of
$G_{H_3}$ on $\mathscr{R}$ is also transitive.

\begin{figure}[htbp]

  \centering
  \includegraphics[height=4.5cm]{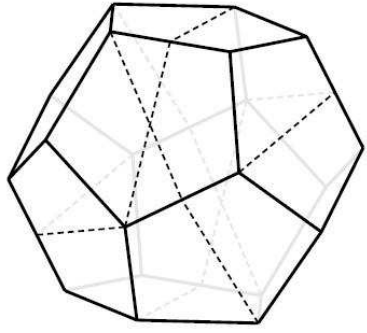}
  \caption{Regular dodecahedron}
 \end{figure}

 It isn't hard to see that,     $|R(i,j)|= 0$  if  $s_i \perp s_j
 $ and $|R(s_i ,s_j )|= 1$ otherwise.

Now consider the algebra $B_{G_{H_3}} (\Upsilon)$.  Since all
elements of $R$ lie in the same conjugacy class, so all $\tau_i $
equal and all $\mu_i$ equal. We denote them as $\tau$ and $\mu$
respectively. Suppose $\mu \neq 0$, in these cases we can set $\mu
=1$ by lemma 5.3.  We have

\begin{lem}
\begin{align*}
&(a) w_0 e_0 e_2 =e_0 e_2 ;   &(b)& w_0 s_0 w_0 ^{-1} = s_2 ;\\
&(c) w_0 ^3 =c ;    &(d)& \mbox{the group generated by}\ \{ s_0 ,
s_2
 , w_0 \}  \  \mbox{has order 24}.\\
&(e) e_{i_{\alpha}} e_{j_{\alpha}} = e_{j_{\alpha}} e_{k_{\alpha} }
= e_{k_{\alpha}} e_{i_{\alpha}}.
\end{align*}

\end{lem}
\begin{proof}
First $(a)$ follows from identities

 $(e_1 e_0 ) e_2 = s_0 s_1 s_0
s_1 s_0 e_0 e_2 =s_0 s_1 s_0 s_1
 e_0 e_2  $ and
$(e_1 e_0 )e_2 = (e_1 e_2 )e_0 =s_1 s_2 s_1 e_2 e_0
 $.

 $(b)$ and $(c)$ follow by direct computations. Denote the group in $(d)$ as $H$.
 Computation shows
 $w_0 s_2 w_0 ^{-1} = c s_0 s_2 $ which implies the order $8$
 abelian subgroup $H^{'} $ generated by $\{ s_0 ,\ s_2 ,\ c \}$ is
 normal in $H$, and the quotient group $H/H^{'}$ consists of $\{ [1] ,\ [w_0 ],\ [w_0 ^2 ]
 \}$. So $(d)$ follows. For $(e)$ we first prove the special case of $\alpha =1$ then the other cases follow by
 conjugating action of $G_{H_3}$. Now the first "=" in $(e)$ is certified by the
 following identity and the second one can be proved similarly. The
 first $'='$ below is by $(a)$.

 $e_0 e_2 = w_0 e_0 e_2 w_0 ^{-1} = w_0 e_0 w_0 ^{-1} w_0 e_2 w_0 ^{-1} = e_2 e_{i_0}.$

\end{proof}

\begin{rem} For $0\leq i\leq 14 $, we define

$G^i = \{ w\in G_{H_3}\  |\  ws_i w^{-1} =e_i \}$  and $H^i =\{ w\in
G_{H_3}\  |\  we_i =e_i \}$.

 For $1\leq \alpha
\leq 5$ define $G_{\alpha}=\{ w\in G_{H_3}\ |\ w (R_{\alpha})
=R_{\alpha} \}$ and
 $H_{\alpha} =\{ w\in G_{H_3}\ |\ we_{i_{\alpha} } e_{j_{\alpha}} =e_{i_{\alpha} } e_{j_{\alpha}}
 \}$.

 There is  $H^i \subset G^i $ and  $H_{\alpha} \subset G_{\alpha}$.
 Since $G_{H_3}$ acts on $R$ transitively and $|R|=15$,  we see if $s_j ,\ s_k$ are the other two
 reflections commuting with $s_i $, then $G^i $ is the order 8 group
 generated by $\{ s_i ,\ s_j ,\ s_k \}$. By $(2)$ of Theorem 7.1 below we
 know $H^i $ is the order 2 group generated by $s_i$.

 Since the action of $G_{H_3}$ on $\mathscr{R}$ is transitive
 we have $|G_{\alpha}|=24$. By $(d)$ of above lemma we have $H_{\alpha} =G_{\alpha
 }$,and that $G_1 $ is generated by $\{ s_0 ,\ s_2 \ , \ w_0 \}$.

\end{rem}

\begin{lem}
$\dim B_{G_{H_3}} (\Upsilon) \leq 1045$.
\end{lem}
\begin{proof}
Recall the set spanning $B_{G_{H_3}}(\Upsilon)$ in lemma 5.2.
Consider an element $e= e_i e_j e_k$, such that $i,j,k$ are
different and every two elements in $\{ i,j,k\}$ satisfy condition
$(4)$ of Definition 1.1. Notice any two of $\{ e_i ,e_j ,e_k \}$
commutes. By above discussion $\{ s_i,s_j,s_k \}$ is some class
$R_{\alpha} $. Now we have $e_{i_{\alpha }} e_{j_{\alpha}}
e_{k_{\alpha}} =e_{j_{\alpha }} e_{k_{\alpha}} e_{k_{\alpha}}= \tau
e_{j_{\alpha }} e_{k_{\alpha}} $, which is by $(e)$ of Lemma 7.1.
  So we have proved that $B_{G_{H_3}}(\Upsilon )$ is spanned by the set

$$\Lambda = G_{H_3} \coprod \{ w e_i \} \coprod \{ w e_i e_j | s_i \perp s_j
\}.$$

The relation $s_i e_i = e_i $ implies $|\{ w e_i \}|\leq
\frac{120}{2}\times 15 =900$. By above discussion and $(e)$ of lemma
7.1, we know there are at most 5 kinds of $e_i e_j $ appearing in
$\{ w e_i e_j | s_i \perp s_j \}$. For every such $e_i e_j$ there is
a group $H_{i,j}$ of order 24 such that $w e_i e_j = e_i e_j $ for
any $w\in H_{i,j}$. So $|\{ w e_i e_j | s_i \perp s_j \}|\leq
\frac{120}{24} \times 5 =25$ and the lemma follows.

\end{proof}
\begin{rem}
The set $\Lambda$ can be presented explicitly as follows. For any
$i$ let $\{ w^i _j \} _{1\leq j\leq  60}$ be a set of
representatives of left cosets of the group $<1,\ s_i >$.  Let $\{
w^{\alpha} _{\beta} \}_{1\leq \beta \leq 5}$ be a set of
representatives of the left cosets of $H_{\alpha } $ (see Remark 7.1
) in $G_{H_3}$. Then $$\Lambda = G_{H_3} \coprod \{ w^i _j e_i
\}_{0\leq i\leq 14; 1\leq j\leq 60} \coprod \{ w^{\alpha} _{\beta}
e_{i_{\alpha}}e_{ j_{\alpha}}\}_{1\leq \alpha \leq 5; 1\leq \beta
\leq 5}.$$
\end{rem}

\nd{\bf Some Irreducible Representations. }

There are four 15 dimensional irreducible representations and one 5
dimensional irreducible representations of $B_{G_{H_3}}(\Upsilon )$
as follows.

The conjugating action of  $G_{H_3}$ on $R$ is transitive. Since
every element of the subgroup $G_0 =<s_0 ,\ s_2 , \ c >$ commutes
with $s_0 $ and $|G_0 |=8$, so $G_0 $ is the stablizing group of
this action at $s_0 $.  $G_0 $ has the following four one
dimensional representations $\{ \sigma _i \}_{0\leq i\leq 3}$ that
sending $s_0 $ to 1.
\begin{align*}
&(1) \sigma _0 (s_0 ,\  s_2 ,\ c)=(1,\ 1,\ 1);   &(2)& \sigma _1
(s_0 ,\
s_2 ,\ c)=(1,\ -1,\ 1);\\
&(3) \sigma _2 (s_0 ,\ s_2 ,\ c)= (1,\ 1,\ -1);  &(4)& \sigma _3
(s_0 ,\ s_2 ,\ c)= (1,\ -1,\ -1).
\end{align*}

For every $0\leq i\leq 3$, we have a left representation of
$G_{H_3}$: $Ind ^{G_{H_3}} _{G_0} (\sigma _i )$ .  They are all 15
dimensional representations whose representation spaces can be
identified with a space $V$ spanned by a basis $\{ v_i \} _{0\leq
i\leq 14}$ in bijection with the set of left cosets $\{ wG_0 \}
_{w\in G_{H_3}}$. The bijection $\phi$ from the second set to the
first one is defined by: if $ws_0 w^{-1} =s_i $ then $\phi (wG_0) =
v_i $.

Now for every $0\leq \alpha \leq 3$, we can extend every $Ind
^{G_{H_3}} _{G_0} (\sigma _{\alpha} )$ to a representation $\rho
_{\alpha}$ of $B_{G_{H_3}}(\Upsilon )$ as follows.
\begin{align*}
&(1) \rho _{\alpha}(e_i ) (v_i ) =\tau v_i . \ \ \ \ \ \ \ \ \ (2)
\rho _{\alpha}(e_i ) (v_j ) = 0 \ \mbox{ if}\   i\neq j \ \mbox{and}
\
s_i \perp s_j .\\
&(3) \rho _{\alpha}(e_i ) (v_j ) = \sigma _{\alpha} (s_k ) (v_j )\
\mbox{if}\ i\neq j ,\  and \  s_i \ \mbox{ isn't perpendicular to} \
s_j \ \mbox{
such that}\ s_k \ \mbox{is the}\\
 & \mbox{ unique reflection satisfying }\  s_k s_i s_k =s_j .
\end{align*}
By definition the operator $ \rho _{\alpha}(e_i )$ is a projector to
the line $\mathbb{C} v_i \subset V$ for all $\alpha$. For every
$0\leq \alpha \leq 3 $, define a $15\times 15$ matrix $M^{\alpha } =
( m^{\alpha} _{i,j} ) $ by the identities $\rho _{\alpha}(e_i ) (v_j
)=m^{\alpha} _{i,j} v_i $. All entries of $M^{\alpha } $ belong to
$\{ \pm 1,\ \tau , \ 0 \}$. By definition the diagonal elements of
all $M^{\alpha }$'s are all $\tau$, and non-diagonal elements are
all constants. So $\det M^{\alpha } $ is a nonzero polynomial of
$\tau$ for all $\alpha$.
\begin{prop}

(1) Above definition of $\rho _{\alpha}(e_i )$'s extends $Ind
^{G_{H_3}} _{G_0} (\sigma _{\alpha} )$ to a representation $\rho
_{\alpha}$ of $B_{G_{H_3}}(\Upsilon )$. (2) The representation $\rho
_{\alpha}$ is irreducible if and only if $\det M^{\alpha } \neq 0$.
(3) $\rho _{\alpha}\ncong \rho _{\beta}$ if $\alpha \neq \beta$.
\end{prop}
\begin{proof}
  Direct computation shows $(1)$. For $(2)$, first we observe $v_i$
  is a generator for any $i$ since the conjugating action of $G_{H_3}$ on $R$
  is transitive. Suppose $ v\in V$ is a nonzero vector. If $\det M^{\alpha }\neq
  0$ then there is some $i$ such that $0\neq \rho _{\alpha} (e_i )(v) \in \mathbb{C} v_i$ by definition of $M^{\alpha }$. So $v$ is a generator and  $\rho
  _{\alpha}$ is irreducible. If  $\det M^{\alpha }= 0$ then the
  space $\cap ^{14} _{i=0} Ker \rho _{\alpha} (e_i ) $ is
  nonzero. It isn't hard to see this subspace is a submodule, thus  $\rho
  _{\alpha}$ is reducible.

  For $(3)$, suppose $\psi : V\rightarrow V$ is an $B_{G_{H_3}}(\Upsilon)$
  isomorphism from $\rho _{\alpha}$ to  $ \rho
  _{\beta}$. By $\psi ( \rho _{\alpha}(e_i )(v) )= \rho _{\beta }(e_i ) (\psi (v))
  $, so $\psi (Im  \rho _{\alpha}(e_i ) ) = Im \rho _{\beta }(e_i )
  $, which implies $\psi (v_i ) = \lambda _i v_i $ for some $ \lambda _i \neq
  0$. Now for $w\in G_0 $, on one hand we have

  $\psi ( \rho _{\alpha}(w )(v_0 ))= \psi (\sigma _{\alpha } (w) v_0)
  = \sigma _{\alpha} (w) \lambda _0 v_0 $,on the other hand

  $\psi ( \rho _{\alpha}(w )(v_0 )) =
\rho _{\beta }(w )(\psi (v_0 ))= \lambda _0 \sigma _{\beta } (w) v_0
. $ So we have $\sigma _{\beta } (w) = \sigma _{\alpha} (w)$ for any
$w\in G_0$, which implies $\alpha =\beta $.

\end{proof}

There is annother irreducible representation related to the action
of $G_{H_3 }$ on $\mathscr{R}$ defined as follows.

Let $U= \mathbb{C} <u_1 ,\cdots ,u_5 >$ be a 5 dimensional vector
space. Define a representation  of $G_{H_3 }$ on $U$ as: $ w(u_i )
=u_j $ if $w (R_i )=R_j $, for $w\in G_{H_3 }$ and $1\leq i\leq 5$.

For $0\leq i\leq 14$, define $[i]\in \{ 1,\cdots ,5 \}$ by the
relation $s_i \in R_{[i]}$. For $0\leq i\leq 14$, $1\leq p\leq 5$,
we set
\[ e_i (v_p)= \left\{
\begin{array}{ll}
\tau v_p , \ &if\ s_i \in R_p ; \\
v_{[i]} ,\ &if \ s_i  \notin  R_p .
\end{array}
\right.
\]

\begin{lem}
Above action of $G_{H_3 }$ and $e_i $'s on $U$ extends to a
representation  $\rho _4$ of $B_{G_{H_3}}(\Upsilon)$. This
representation is irreducible if and only if $(\tau -1)^4 (\tau
+4)\neq 0$.
\end{lem}
\begin{proof}
The first claim can be proved by direct computations. For the second
one we first observe $\rho _4 (e_i )$ is a projector to $\mathbb{C}
u_{[i]}$ and $\rho _4 (e_i ) = \rho _4 (e_j )$ if $i,\ j$ lie in the
same equivalent class. So for every $1\leq p\leq 5$ we have a well
defined projector $J_p $(onto $\mathbb{C} u_p$) by setting $J_p =
\rho _4 (e_i )$ for any $i\in R_p$. Define a $5\times 5$ matrix $M^4
= (m^4 _{p,q})$ by setting $J_p (u_q ) = m^4 _{p,q} u_p $. This
matrix is clear: diagonal entries are all $\tau$ and non-diagonal
entries are all $1$. So $\det M^4 =(\tau -1)^4 (\tau +4)$.
 An argument similar to Proposition 7.1 shows $\rho _4$ is
 irreducible if and only if $\det  M^4 \neq 0$.

\end{proof}

\begin{thm}
(1)  If  $ \Pi ^4 _{p=0}\det M^p \neq 0$, $B_{G_{H_3}}(\Upsilon )$
is a 1045 dimensional semisimple algebra and have $\Lambda $ (remark
7.2 ) as a basis. Notice  $ \Pi ^4 _{p=0}\det M^p $ is a polynomial
of $\tau $.

(2) For all $\tau $, $B_{G_{H_3}}(\Upsilon )$ is a 1045 dimensional
algebra having $\Lambda $ as a basis.
\end{thm}
\begin{proof}
Suppose $\rho _5 $, $\rho _6$, $\cdots$, $\rho _N $ are all
irreducible representations of $B_{G_{H_3}}(\Upsilon )$ induced by
the quotient map $\pi : B_{G_{H_3}}(\Upsilon )\rightarrow \mathbb{C}
G_{H_3}$. They are different from $\rho _0 ,\rho _1 , \cdots $, and
 $\rho _4$ because $e_i $'s act as zero on them. We have $\sum _{i=5} ^N \dim \rho _i ^2 =|G_{H_3}|=120.$

 If $ \Pi ^4 _{p=0}\det M^p \neq 0$ then $\rho _1 ,\rho _2 , \cdots $, and
 $\rho _4$ are all irreducible.  In these cases by  Wedderburn-Artin theorem we
have

 $\dim B_{G_{H_3}}(\Upsilon ) \geq \sum _{i=5} ^N \dim \rho _i ^2 + \sum _{i=0} ^3 \dim \rho _i ^2 + \dim \rho _4 ^2
 = 120 + 900+ 25 =1045.$

Combining  with lemma 7.2 we have proved (1).

For later convenience , in the set $\Lambda$ we denote elements of
$G_{H_3}$ as $x_1 , \cdots ,x_{120} $, denote $e_1 , \cdots ,e_{15}
$ as $x_{121} , \cdots ,x_{135}$. Denote other elements of $\Lambda$
as $x_{136} ,\cdots ,x_{1045}$. Suppose first $ \Pi ^4 _{p=0}\det
M^p \neq 0$. In these case since $\{ x_i \}_{1\leq i\leq 1045 }$ is
a basis of $B_{G_{H_3}}(\Upsilon )$ , every product $x_i x_j$ can be
expanded uniquely as a linear sum of $x_k $'s: $(*)$  $x_i x_j =
\sum _{k=1} ^{1045} f_{i,j} ^k (\tau ) x_k $.

We observe the following facts.

(a)  $f_{i,j} ^k (\tau )$'s are all polynomial functions of $\tau $;

(b)  The identity $(*)$ actually holds in $B_{G_{H_3}}(\Upsilon )$
for all $\tau$'s.

Let $A = \mathbb{C} < X_1 , \cdots ,X_{1045}>$ be a vector space
with a basis $\{ X_1 , \cdots X_{1045} \}$. Define a product on $A$
by setting $X_i X_j =\sum _{k=1} ^{1045} f_{i,j} ^k (\tau ) X_k $.
This product make $A$ into an associative algebra when  $ \Pi ^4
_{p=0}\det M^p \neq 0$. Combining with (a) it follows that this
product is well-defined and making $A$ into an associative algebra
for all $\tau$'s. Denote this algebra as $A(\Upsilon )$. Recall we
have argued that in case of $H_3$ the data $\Upsilon$ consists of
one term $\tau $ essentially. A simple check of this product shows:

 (c) $A(\Upsilon )$ is generated by $\{ X_1 ,\cdots ,X_{135} \}$
 for all $\tau$'s.

 (d) The map $[ x_i \mapsto X_i \ for \ 1\leq i\leq 135 ]$ extends
 to a morphism $\phi :B_{G_{H_3}}(\Upsilon )\rightarrow A(\Upsilon
 )$ for all $\tau$'s.

 By (c) the morphism $\phi$ is surjective. Then by lemma 7.2 $\phi$ is an
 isomorphism and (2) follows.

\end{proof}

\nd{\bf Cellular structures.} let $(\Lambda , M , C ,* )$ be the
cellular structure of $\mathbb{C} B_{G_{H_3}}(\Upsilon ) $. The
algebra $B_{G_{H_3}}(\Upsilon)$ has a cellular structure $(\bar
{\Lambda },\bar { M}, \bar {C} ,*)$ as follows.

\begin{itemize}
  \item $\bar{\Lambda}=\Lambda \cup \{ \lambda _0 , \lambda _1 ,\cdots , \lambda _4 \} .
  $ Extend the partial order in $\Lambda$ by setting: $\lambda _i \prec \lambda
  $ for any $0\leq i\leq 4$ and any $\lambda \in \Lambda$; $\lambda _4 \prec \lambda _i
  $ for $0\leq i\leq 3$.
  \item Let $*$ be the involution defined in lemma 5.5.
  \item For $\lambda \in \Lambda , \bar{M}(\lambda ) = M(\lambda )$;
  for $0\leq i\leq 3$, $\bar{M}(\lambda _i )=\{ 0,1,\cdots ,14 \}$;
  $\bar{M}(\lambda _4 ) =\{ 1,2,\cdots ,5 \}$.
  \item For $\lambda \in \Lambda $, $S,T
\in \bar{M} (\lambda )$, let $\bar{C} ^{\lambda } _{S,T}
=j(C^{\lambda } _{S,T})$. Where $j$ is the naturally injection from
$\mathbb{C} G_{H_3}$ to $B_{G_{H_3}}(\Upsilon)$.
  \item For every $0\leq i\leq 14$, choose $w_i \in G_{H_3}$ such that
  $w_i e_0 w_i ^{-1} = e_i $ and $w_0 = id  $.   Set $J_0 =(1+s_2 + s_2 c + c) ,\  J_1 = (1-s_2 +c-cs_2 ),\  J_2
  =(1+s_2  - cs_2 -c ),\ J_3 =(1-s_2 +cs_2 -c)$, which are
  idempotents of the group algebra $\mathbb{C}<s_2 , c>$
  corresponding to $\sigma _0 , \cdots ,\sigma _3 $.
Then for $0\leq \alpha \leq 3$ and  $0\leq i ,j \leq 14$, set
$\bar{C}^{\lambda_{\alpha}} _{i,j} = w_i J_{\alpha} e_0 w_j ^{-1}.$
  \item As before suppose $R_{\alpha} =\{ i_{\alpha} ,j_{\alpha}, k_{\alpha}
  \}$. For $1\leq \alpha ,\beta \leq 5$ choose $w^{\alpha} _{\beta}\in W_{H_3}$
  such that $ w^{\alpha } _{\beta} (R_{\alpha} )= R_{\beta}$. Where
  the conjugating action is defined in the beginning of this
  section.Then
  set $\bar{C} ^{\lambda _4 } _{\beta ,\alpha } = w^{\alpha} _{\beta} e_{i_{\alpha}}e_{j_{\alpha }}. $
\end{itemize}

\begin{thm}
Above data define a cellular structure on $B_{G_{H_3}}(\Upsilon)$.
\end{thm}
\begin{proof}
$(C1)$ follows from Theorem 7.1. $(C2)$ is proved by the following
identities.

$*(\bar{C}^{\lambda_{\alpha}} _{i,j}) = *(w_j
^{-1})*(e_0)*(J^{\alpha})
*(w_i ) = w_j e_0 J^{\alpha} w_i ^{-1} =\bar{C}^{\lambda_{\alpha}} _{j,i} $;

$*(\bar{C}^{\lambda _4} _{\beta ,\alpha })=
*(e_{j_{\alpha}})*(e_{i_{\alpha}})*(w^{\alpha} _{\beta})= e_{j_{\alpha}} e_{i_{\alpha}} (w^{\alpha} _{\beta}) ^{-1} =
 (w^{\alpha} _{\beta}) ^{-1} e_{i_{\beta} } e_{j_{\beta}} = w^{\beta } _{\alpha} h e_{i_{\beta}}
 e_{j_{\beta}}=\bar{C}^{\lambda _4} _{\alpha ,\beta }$. Where the
 third "=" is because $w^{\alpha} _{\beta} (R_{\alpha} )=R_{\beta}$
 and $e_{i_{\beta}}e_{j_{\beta}}=e_{j_{\beta}}e_{k_{\beta}}=e_{k_{\beta}}e_{i_{\beta}}.
 $ By Remark 7.1 , there is some $h\in G_{\beta}$ such that the fourth
 "=" holds. Also by Remark 7.1 we get the fifth "=". $(C3)$ in cases
 of $\lambda _i (0\leq i\leq 3)$ are proved by the following
 identities.

 $(1)$ $w \bar{C}^{\alpha} _{i,j} = ww_i J^{\alpha}e_0 w^{-1} _j = w_k (w^{-1} _k w w_i)J^{\alpha}e_0 w^{-1} _j
 =\sigma _{\alpha} (w^{-1} _k w w_i) \bar{C}^{\alpha} _{k,j}$. Where
 $k$ is determined by $w^{-1} _k w w_i \in G_0 $. The last "=" is
 because $v J^{\alpha} = \sigma _{\alpha} (v) J^{\alpha}$ for $v\in G_0
 $.

 $(2)$ $e_l \bar{C}^{\alpha} _{i,j} = e_l w_i J^{\alpha}e_0 w^{-1} _j =
e_l e_i w_i J^{\alpha} w^{-1} _j $
\[
  = \left\{
 \begin{array}{ll}
    0 \ mod(I ^{< \lambda _{\alpha}})  &\mbox{ when $s_l \perp s_i$ ;}\\
    s_{l,i} e_i w_i J^{\alpha} w^{-1} _j =s_{l,i} w_i J^{\alpha} e_0
    w^{-1} _j = \\
    \sigma _{\alpha} (w^{-1} _k s_{l,i} w_i ) \bar{C}^{\alpha} _{k,j} mod(I ^{< \lambda
    _{\alpha}})
    &\mbox{when $s_l$ isn't perpendicular to $s_i$ .}
\end{array}
\right.
\]

where $s_{l,i}$ is the unique reflection such that $s_{l,i }s_i
    s_{l,i}=s_l$
,  $k$ is determined by $w^{-1} _k s_{l,i} w_i \in G_0$, and $I ^{<
\lambda _{\alpha}}$ is the ideal generated by $\{ \bar{C}^{\lambda }
_{i,j} \} _{\lambda < \lambda _{\alpha}}.$

In the case of $(C4)$ we have

$(3)$ $w \bar{C}^{\lambda _4} _{\beta ,\alpha} = w w^{\alpha}
_{\beta} e_{i_{\alpha}} e_{j_{\alpha}}= \bar{C}^{\lambda _4}
_{\gamma ,\alpha} .$ Where $\gamma$ is determined by $ww^{\alpha}
_{\beta} (R_{\alpha}) =R_{\gamma}$. The last "=" is by using remark
7.1.

$(4)$ \[ e_i \bar{C}^{\lambda _4} _{\beta ,\alpha} = e_i
e_{i_{\beta}} e_{j_{\beta}} w^{\alpha} _{\beta} = \left\{
\begin{array}{ll} \tau \bar{C}^{\lambda _4} _{\beta ,\alpha} &\mbox{
when $i\in \{ i_{\beta} ,j_{\beta} ,k_{\beta}
\} $. }\\
s_{i,i_{\beta}} e_{i_{\beta}} e_{j_{\beta}} w^{\alpha} _{\beta}
=\bar{C}^{\lambda _4} _{\gamma ,\alpha} &\mbox{ otherwise. }
\end{array}
\right.
\]
Where $\gamma$ is determined  by $i\in R_{\gamma}$.
\end{proof}

\section{Canonical Presentations  }

\subsection{Real cases }

We define an algebra $B^{'} _G  (\Upsilon )$ with certain canonical
presentation when $G$ is a Coxeter group or a cyclotomic reflection
group of type $G(m,1,n)$, then prove $B^{'} _G (\Upsilon )$ is
isomorphic to $B _G (\Upsilon )$. First we do it in cases of
dihedral groups.

\begin{defi}

The algebra $B^{'} _{G } (\Upsilon)$ have the following presentation
when $G$ is $G_n$,  the dihedral group of type $I_2 (n )$.
$$TABLE\ 3.\ Presentation\  for\  B _{G_{n}} (\Upsilon ).$$

\begin{tabular}{|l|l|l|}

\hline

&  $B_{G_{2k+1}}(\Upsilon) $       &   $B_{G_{2k}}(\Upsilon)$ \\

\hline

generators &  $S_0 ,S_1 , E_0 ,E_1 $   &  $S_0 ,S_1 , E_0 ,E_1 $ \\

\hline

relations  &  $1)$ $[S_0 S_1 \cdots ]_{2k+1} =[S_1 S_0 \cdots
]_{2k+1} $ ; & $1)$ $[S_0 S_1 \cdots ]_{2k} =[S_1 S_0 \cdots ]_{2k}
; $ \\

    & $2)$ $S_0 ^2 =S_1 ^2 =1 $;  & $2)$ $S_0 ^2 =S_1 ^2 =1 $;   \\
    & $3)$ $S_i E_i = E_i = E_i S_i $ for $i=0,1$;  & $3)$ $S_i E_i = E_i = E_i S_i $ for $i=0,1$; \\
     & $4)$ $E_i ^2 = \tau E_i $ for $i=0,1 $;  & $4)$ $E_i ^2 = \tau_i E_i $ for $i=0,1 $; \\
    &  $5)$ $E_0 [S_1 S_0 \cdots ]_{2i-1 } E_0$ & $5)$ $E_{0} [S_1 S_0 \cdots ]_{2i-1 } E_{0}$ \\
    &$=\mu E_0 $ for $1\leq i\leq
k$; & $= (\mu _i +\mu _{i+k}
 ) E_0 $ for $1\leq i\leq k$; \\
    & $6)$ $E_1 [S_0 S_1 \cdots ]_{2i-1 } E_1 $ & $6)$ $E_1 [S_1 S_0 \cdots ]_{2i-1} E_1 $   \\
     & $=\mu E_1 $ for $1\leq i\leq
k$;  & $= (\mu _i +\mu
_{i+k}  )E_1 $ for $1\leq i\leq k$;\\
    & $7)$ $[S_0 S_1 \cdots ]_{2k } E_0 $& $7)$ $[S_1 S_0 \cdots ]_{2k-1} E_0 $  \\
    & $= E_1 [S_0 S_1 \cdots ]_{2k } $ ; & $ = E_0 [S_1 S_0 \cdots ]_{2k-1}=
E_0$; \\
 & $8)$ $[S_1 S_0 \cdots ]_{2k } E_1 $ & $8)$ $[S_0 S_1 \cdots ]_{2k-1} E_1 $ \\
  & $= E_0 [S_1 S_0 \cdots ]_{2k }.$ &$= E_1 [S_0 S_1 \cdots
]_{2k-1}=E_1$ ;  \\
 & & $9)$ $E_1 W E_0 = E_0 W E_1 =0 $. \\

 \hline

 \end{tabular}\\

 Where in $9)$ $W$ is any element composed by $\{ S_0 , S_1 \}$.

\end{defi}

\begin{thm}
If $G$ is a dihedral group, then $B_G (\Upsilon )$ is isomorphic to
$B^{'} _G (\Upsilon )$.

\end{thm}

\begin{proof}

We consider the cases when $G$ is of type $I_2 (2k) $. The cases for
$G$ of type $I_2 (2k+1)$ are similar and easier.

Denote the algebra $B_{G } (\Upsilon)$ ,$B^{'} _G (\Upsilon )$  as
$B$, $B ^{'}$ respectively.  Let $j$ be the morphism from
$\mathbb{C} G $ to $B^{'} $ by mapping $s_i \in G $ to $S_i \in
B^{'}$ for $i= 0,1$. Let $\pi $ be the morphism from $B^{'}$ to
$\mathbb{C} G $ by mapping $S_i \in B^{'}$ to $s_i \in G $, $E_i $
to 0. There is $\pi \circ j =id _{\mathbb{C} G }$, which implies
that $j$ is injective. For saving notations we denote $j(w)$ as $w$
for $w\in G$.

For $2\leq 2i \leq 2k-2 $, choose any $w \in G $ such that $s_{2i} =
ws_0 w^{-1}$ and let $E_{2i} = w E_0 w^{-1} $. $E_{2i} $ is well
defined with no dependence on choice of $w$ (a special case of Lemma
8.1 later ). For example, choose $w=[S_1 S_0 \cdots ]_{2i-1} $ so
$E_{2i} = [S_1 S_0 \cdots ] _{2i-1} E_0 [S_1 S_0 \cdots ] _{2i-1}$.
Similarly for $3\leq 2i-1 \leq 2k-1 $, define $E_{2i-1} = [S_1 S_0
\cdots ]_{2i-2} E_1 [S_1 S_0 \cdots ]_{2i-2} = [S_1 S_0 \cdots
]_{2i-1} E_1 [S_1 S_0 \cdots ]_{2i-1} $. Define a map $\phi $ from
the set of generators of $B$ to $B^{'}$ as : $\phi (T_w ) =w
(=j(w))$; $\phi (e_i) = E_i $ for $0\leq i \leq 2k-1 $. Then $\phi $
extends to a morphism from $B$ to $B^{'}$.  To prove it we only need
to certify that $\phi $ keep all the relations in Definition 1.1.
The case of relation $(0)$ is straightforward. Relation $(1)$ is by
$3)$ in Definition 8.1 of $B_{G_{2k}}(\Upsilon)$ ; $(1)^{'}$ is by
$7) ,8)$; $(2)$ is by $4)$; $3)$ is by later Lemma 8.1; case of $4)$
doesn't arise here; $(6)$ is by $9)$; $5)$ is by the following
computations. First consider the relation for $e_{2i} e_0$. We have

$\phi (e_{2i}) \phi ( e_0 ) = E_{2i} E_0 = [s_1 s_0 \cdots ]_{2i-1}
E_0 [s_1 s_0 \cdots ]_{2i-1} E_0 =s_i E_0 s_i E_0 = s_i (\mu _i +
\mu _{[i+k]} c) E_0  \\
= (\mu _i s_i + \mu _{[i+k]} s_{[i+k]})E_{0} = \phi ( (\mu
_{k(2i,0)} s_{k(2i,0)} + \mu _{k(2i,0)^{'} } s_{k(2i,0) ^{'}} ) e_0
)=\phi (e_{2i} e_0 ).$

Relations for other $e_{2i} e_{2j}$ can be obtained by suitable
conjugating action of $G$ on above equation. The relations for
$e_{2i+1} e_{2j+1} $ are similar.

There is a natural morphism $\psi : B^{'} \rightarrow B$ by
extending the correspondence $S_0 \mapsto s_0 $, $S_1 \mapsto s_1 $,
$E_0 \mapsto e_0 $, $E_1 \mapsto e_1 $.  The fact that $\psi$ keep
relations $1), 2)$ of Definition 8.1 is by $(1)$ of Definition 1.1;
$3)$ is by $(1)$; $4)$ is by $(2)$; $5), 6)$ are by $(5)$; $7),8)$
are by $(1)^{'}$; $9)$ is by $(6)$.

\end{proof}

Suppose $G_M $ is a finite Coxeter group with Coxeter matrix $M=
(m_{i,j})_{n\times n} $. The group $G_M $ has the following
presentation:
$$< s_1 ,s_2 ,\cdots ,s_n \ |[s_i s_j \cdots ]_{m_{i,j}} =[s_j s_i\cdots]_{m_{i,j}}\  for\  i\neq
j;  s_i ^2 =1 \ for\  any\  i> .$$

It is well-known that $G_M $ can be realized as a group generated by
reflections in some $n$ dimensional linear space through cetain
geometric representation  $\rho : G\rightarrow GL(V)$. We identify
$G_M$ with its image in $GL(V)$, denote $\rho (s_i )$ as $s_i$.
Since $G_M$ is real, the index set of reflection hyperplanes $P$ are
in one to one correspondence with the set of reflections $R$. So it
is convenient to denote the reflection hyperplane of $s\in R $ as
$H_s $ and write $e_i $ in the Definition 1.1 as $e_s $. In the
following we denote $G_M $ as $G$. For $w\in G$, any expression $w=
s_{i_1} s_{i_2 }\cdots s_{i_r }$ with minimal length is called a
reduced form of $w$, and define the length of $w$ as $l(w) =r$.
Above definition of $B^{'}_G (\Upsilon)$ when $G$ is a dihedral
group invoke the following definition of $B^{'}_{G_M} (\Upsilon)$.

\begin{defi}
For any Coxeter matrix $M= (m_{i,j})_{n\times n}$, the algebra
$B^{'}_{G_M} (\Upsilon)$ is defined as follows. Denote $\tau_{s_i}$
in $\Upsilon$ as $\tau_i$. If we don't give range for an index then
it means "for all". The generators are $S_1 ,\cdots ,S_n ,E_1
,\cdots ,E_n $. The relations are \begin{align*}
&1) S_i ^2 =1 ; &8)&  E_i w E_j =0  \mbox{ for any word } w  \\
& 2) [S_i S_j \cdots ]_{m_{i,j }} =[S_j S_i \cdots ]_{m_{i,j}}; & &
\mbox{composed from } \{S_i ,S_j
\} \mbox{If } m_{i,j}= 2k>2 ; \\
& 3)  S_i E_i =E_i =E_i S_i ;  & 9)&  E_i [S_j S_i \cdots ]_{2l-1 }
E_i = (\mu _s +\mu _{s^{'}} )
E_i  \\
& 4)  E_i ^2 = \tau _i E_i ;  & &\mbox{for }  1\leq l\leq k, \mbox{If } m_{i,j}= 2k>2 .  \\
&  5)   S_i E_j =E_j S_i  \mbox{if } m_{i,j } = 2 ; &10)& E_i [S_j S_i \cdots ]_{2l-1 } E_i =\mu_{s_{\epsilon }} E_i  \\
& 6)   E_i E_j =E_j E_i  \mbox{if } m_{i,j } = 2 ; & &\mbox{for } 1\leq l\leq k, \mbox{if } m_{i,j} = 2k+1 ;  \\
& 7)  [S_j S_i \cdots ]_{2k -1} E_i = E_i [S_j S_i \cdots ]_{2k
-1} =E_i ,   & &\mbox{Where } \epsilon=i(j) \mbox{if } l \mbox{is odd (even) }.  \\
& \mbox{if } m_{i,j}=2k>2 ;  &11)& [S_i S_j \cdots ]_{2k } E_i = E_j
[S_i S_j \cdots ]_{2k }
\mbox{If } m_{i,j} = 2k+1 .  \\
\end{align*}

\end{defi}

\begin{rem} If $M$ is irreducible and of simply laced type , i.e, $m_{i,j} \in \{ 1,2,3
\}$, we can set all $\mu_s =1$ by Lemma 5.3, so above definition
coincide with the definition of simply laced Brauer algebras in
\cite{CFW}.  Above definition includes the case $m_{i,j}=\infty$: in
that case there are no other relations between $S_i ,S_j ,E_i ,E_j$
except $(1),(3),(4)$.
\end{rem}

Let $M$ be as above.  The Artin group $A_M $ has the following
presentation.
$$<\sigma _1 , \sigma _2 ,\cdots , \sigma _n |\ [\sigma _i \sigma _j \cdots  ]_{m_{i,j}} =[\sigma _j \sigma _i
\cdots ]_{m_{i,j}}\ for\  i\neq j >.$$ Here we denote $A_M $ as $A$.
Let $A^{+} $ be the monoid generated with the same set of generators
and relations. Let $J : A^{+} \rightarrow A $ be the natural
morphism of of monoids. It is proved that $J$ is injective for all
Artin groups Garside$\cite{Ga}$ Brieskorn-Saito $\cite{BS}$ Paris
$\cite{Pa}$ . The following theorem is well known.
\begin{thm}
For any $w\in G $, suppose $l(w) =r $ and let $s_{i_1 } \cdots
s_{i_r } $ and $s_{j_1 } \cdots s_{j_r }$ be two reduced forms of
$w$, then in $A^{+}$ we have $\sigma _{i_1} \cdots \sigma _{i_r } =
\sigma _{j_1} \cdots  \sigma _{j_r } $.

\end{thm}
So there is a well defined injective map $\tau : G \rightarrow
A^{+}$ as follows. For $w\in G$, let $s_{i_1 } \cdots s_{i_k } $ be
a reduced form of $w$ and let $\tau (w) = \sigma _{i_1 } \cdots
\sigma _{i_k }$. Denote the natural map from $A^{+}$ to $G$
extending $\sigma _i \mapsto s_i $ as $\pi $. In $A^{+}$ we denote
$b \prec c $ if there is $a \in A^{+}$ such that $ab=c $. This
define a partial order for $A^{+}$. Here is an important result in
Artin group theory.

\begin{thm}[\cite{BS}, \cite{Ga}]
For $a\in A^{+}$, if $\sigma _i \prec a$, $\sigma _j \prec a$, then
$[\cdots \sigma _j \sigma _i ]_{m_{i,j}} \prec a$.

\end{thm}

Now we can prove the following lemma.

\begin{lem}
Suppose $G$ acts on a set $S$. Suppose a subset $\{v_1 ,\cdots , v_n
\} \subset S$ satisfy \begin{align*} &(1) If\ m_{i,j}=2k+1 ,\ then\
[s_i s_j \cdots ]_{2k} (v_i )=v_j ;
&(3)& If\  m_{i,j} =2 ,\  then\  s_i (v_j ) =v_j ; \\
&(2) If\  m_{i,j} =2k ,\  then\  [s_i s_j \cdots ]_{2k-1} (v_j )
=v_j ; &(4)& s_i (v_i ) = v_i .
\end{align*}
Then an identity $w s_i w^{-1 } = s_j $ in $G$ implies $w (v_i ) =
v_j $.

\end{lem}

\begin{proof}
We prove it by induction on $l(w)$. When $l(w) =0 $ it is evident.
Suppose the lemma is true when $l(w) < k $ and suppose we have an
identity $w s_i w^{-1} =s_j $ where $l(w) =k$.  If $l(w s_i ) = l(w)
-1$, let $w^{'} = ws_i $. Since $w^{'} s_i (w^{'} )^{-1}  =w s_i
w^{-1} = s_j $, by induction we have $w^{'} (v_i ) = v_j $. Which
implies $w(v_i ) =  w^{'} (v_i ) = v_j $ by (4). Now suppose $l(ws_i
) = l(w) +1 $. Let $s_{i_1 } \cdots s_{i_k }$ be a reduced form of
$w$. We have $s_{i_1 } \cdots s_{i_k } s_i = s_j s_{i_1 } \cdots
s_{i_k }$. Because both sides of the identity are reduced forms, by
Theorem 8.2 we have $\sigma _{i_1 } \cdots \sigma _{i_k } \sigma _i
= \sigma _j \sigma _{i_1 } \cdots \sigma _{i_k } =\tau (ws_i)$. From
the condition $l(ws_i ) = l(w) +1$ we know $i_k \neq i$, so by
Theorem 8.3 we have $[\cdots \sigma _{i_k } \sigma _i ] _{m_{i_k
,i}} \prec \tau (ws_i )$. So $\tau (ws_i ) = a [\cdots \sigma _{i_k
} \sigma _i ] _{m_{i_k ,i}} $ for some $a \in A^{+}$. Denote $\pi
(a) $ as $w^{'}$ ,and $\pi ([\cdots  \sigma _i \sigma _{i_k } ]
_{m_{i_k ,i} -1})$ as $u$. So $w= w^{'} u $. An argument of length
shows $l(w^{'}) = l(w) -l(u) $. There is $u s_i u^{-1} = s_{i_k }$,
and by (1), (2) , (3) we have $u(v_i ) = v_{i_k }$.  So $w^{'}
s_{i_k } (w^{'})^{-1} = w s_i w^{-1} = s_j $. By induction we have
$w^{'} (v_{i_k } ) = v_j $ which implies $w (v_i ) = w^{'} u (v_i )
= v_j $.

\end{proof}
\begin{thm} When $G_M$ is finite then $B_{G_M } (\Upsilon)\cong B_{G_M }^{'}
(\Upsilon)$.

\end{thm}
The proof is still by constructing  a morphism from $B_{G }^{\prime
 }
(\Upsilon)$ to $B_{G } (\Upsilon)$ and a morphism back, then show
they are the inverse of each other. The following lemma is
well-known Humphreys$\cite{Hu}$.
\begin{lem} Suppose $G$ is a finite Coxeter group. Then

$(1)$ For any $i,j$, if $w\in G$ fix $H_{s_i} \cap H_{s_j }$
point-wise, then $w$ lies in the subgroup generated by $s_i , s_j$.

$(2)$ For any two different $s, s^{'} \in R$, there are $w\in G,
1\leq i<j\leq n$ such that $w(H_s \cap H_{s^{'}}) = H_{s_i } \cap
H_{s_j }$. So $wsw^{-1} $
 and $ws^{'}w^{-1} $ lie in the subgroup generated by $s_i
 $ and $s_j$.
\end{lem}
\begin{lem} Define a map $\phi$ by setting $\phi (S_i ) =s_i \in G $,$\phi (E_i ) = e_{s_i }$ for $1\leq i\leq
n$, then $\phi$ extends to a morphism from $B_{G }^{\prime }
(\Upsilon)$ to $B_{G } (\Upsilon)$.
\end{lem}
\begin{proof}We need to certify that $\phi $ preserves  all relations of in
Definition 8.2. The facts $\phi$ preserves $1)\sim  8)$ and $11)$
are easy to see.
 Notice in $9) $ and $ 10)$ only two indices $i,j $
are involved.  So we can use lemma 5.4 and  lemma 8.2 to reduce
these cases to dihedral cases, which are proved in Theorem 8.1.
\end{proof}
It isn't hard to see that the morphism $J$ from $\mathbb{C} G $ to
$B_{G }^{\prime } (\Upsilon)$ by sending $s_i $ to $S_i $ is
injective. So for $w\in G$ we can identify $J(w)$ with $w $ for
convenience. Denote the imbedding image of $R$ in $B_{G }^{\prime }
(\Upsilon)$ as $R^{\prime }$. Let $E^{\prime } = \{wE_i w^{-1}
\}_{w\in G, 1\leq i\leq n }$, $E =\{e_s \}_{s\in R}$. By definition
of $B_{G }^{\prime } (\Upsilon)$, the conjugating action of $G$ on
$E^{\prime }$ satisfies conditions in Lemma 8.1. So the map $e_i
\mapsto E_i $ $(1\leq i\leq n)$ extends uniquely to a
$G$-equivariant surjective map  $\varphi :E \rightarrow E^{\prime}
$. Define a map $\psi : E\cup R \rightarrow E^{\prime }\cup
R^{\prime } $ by $\psi (e_s ) = \varphi (e_s )$, $\psi (s)=s$. We
have the following lemma.

\begin{lem}
The map $\psi $ extends to a morphism from $B_{G } (\Upsilon)$ to
$B_{G }^{\prime } (\Upsilon)$. Still denote it as $\psi$.
\end{lem}

\begin{proof}
We need to show that $\psi$ satisfies all relations in Definition
5.1. The cases for relation (0),(1) and (2) is evident. The case for
relation $(1)^{'}$ follows from relation $(4)$ of Definition 8.3.
Case of relation (3) is by definition of $\varphi$. For relations
(3) to (6), we can reduce these cases by (2) of lemma 8.2 and lemma
5.4 to cases of dihedral groups, which are proved in Theorem 8.1.

\end{proof}

By definition $\psi$ and $\phi$ are apparently the inverse of each
other, so Theorem 8.4 is proved.

\subsection{The Cyclotomic $G(m,1,n)$ cases}

Let $G$ be the cyclotomic pseudo reflection group of type
$G(m,1,n)$. As in BMR\cite{BMR}, let $V $ be a $n$-dimensional
complex linear space with a positive definite Hermitian metric $<,>$
, let $\{v_1 ,\cdots ,v_n \}$ be a orthonormal base. Then $G$ can be
imbedded in $U(V)$. It's image consists of monomial matrices whose
nonzero entries are $m$-th roots of unit. Here we give a concise
description of some facts of $G$ without proof.  Suppose $(z_1 ,
\cdots ,z_n )$ is the coordinate system corresponding to $\{v_1
,\cdots ,v_n \}$. Let $\xi = \exp ( \frac {2\pi \sqrt{-1}}{m})$. For
$i\neq j $, $0\leq a\leq m-1 $, define $H_{i,j;a} = \ker (z_i - \xi
^a z_j )= (v_i -\xi ^a v_j )^{\perp }$. Define $H_i =\ker (z_i )
=(v_i ) ^{\perp }$. Then $H_{i,j;a} =H_{j,i;-a}$.

Let $s _{i,j;a }\in U(V)$ be the unique reflection fixing every
points in $H_{i,j;a} $. Let $s_i $ be the pseudo reflection defined
by: $s_i (v_i )=\xi v_i $; $s_i (v_j ) =v_j$ for $j\neq i$. Then the
set $\mathcal {A}$ of reflection hyperplanes of $G$ is $\{ H _{i,j;a
} \} _{i<j ; 0\leq a\leq m-1 } \cup \{ H_i \} _{1\leq i\leq n }$.
The set of pseudo reflections $R$ of $G$ is
$$\{ s _{i,j ;a } \} _{i<j ;0\leq a\leq m-1 } \amalg (\amalg _{k=1} ^{m-1} \{ s^k _i \}_{1\leq i\leq n } ).$$

Above notation gives a decomposition of $R$ into conjugacy classes.
Now we have a look at the algebra $B_G (\Upsilon) $. The data
$\Upsilon $ essentially consists of $\mu , \mu _1 , \cdots , \mu
_{m-1} , \tau _0 ,\tau _1$. Where $\mu _{s_{i,j;a }} =\mu $, $\mu
_{s^k _i } = \mu _k $; $\tau _{H _{i,j;a }} = \tau _1 $, $\tau _{H_i
} =\tau _0 $. As in the real case, we define the following algebra
$B^{'} _G (\Upsilon )$ with canonical presentation.
\begin{defi}
The algebra $B^{'} _G (\Upsilon )$ is generated by  $S_0 , S_1
,\cdots ,S_{n-1 } $, $E_0 , E_1 \cdots , E_{n-1}$ with the following
relations. Where in $15)$ $W$ is any word composed from $S_0 , S_1$.
\begin{align*} & 1)S_0 ^m = S_i ^2 =id ( 1\leq i\leq n-1 ) ; &2) &
S_0 S_1 S_0 S_1
=S_1 S_0 S_1 S_0 ; \\
&3) S_i S_j =S_j S_i  ( \mid i-j \mid \geq 2);  & 4) & S_i S_{i+1}
S_i =
S_{i+1} S_i S_{i+1}  ( 1\leq i\leq n-2 ); \\
&5) E_0 ^2 = \tau _0 E_0 ; &6)& E_i ^2 = \tau _1 E_i  ( 1\leq i\leq n-1 );  \\
&7) S_1 S_0  S_1 E_0 =E_0 S_1  S_0  S_1 = E_0 ;  &8)& S_i S_{i+1} E_i = E_{i+1} S_i S_{i+1} (1\leq i\leq n-2) \\
&9) S_i E_j = E_j S_i  (\mid i-j \mid \geq 2 );  &10)&  (S_0 )^i S_1
(S_0 )^i (E_1 )=E_1 (S_0 )^i S_1 (S_0 )^i
(0\leq i\leq m-1);  \\
&11) S_i E_i = E_i = E_i S_i  (0\leq i\leq n-1 );  &12)& E_1 S_0 ^i
E_1 = \mu _i E_1 (0\leq
i\leq m-1)  \\
&13) E_0 S_1
E_0 = (m-1)\mu E_0 ;  &14)& E_i E_j =E_j E_i (|i-j|\geq 2);\\
&15) E_0 W E_1 = E_1 W E_0 =0 ; &16)&  E_i E_{i+1} =
\mu S_i S_{i+1} S_i E_{i+1} (1\leq i\leq n-2). \\
\end{align*}
\end{defi}
\begin{thm}
The algebra $B^{'} _G (\Upsilon ) $ is isomorphic to $B_G (\Upsilon
) $.

\end{thm}
\begin{proof}The strategy of proof of this theorem is the same as for the real
case. We construct a morphism $\Phi $ from $B^{'} _G (\Upsilon ) $
to $B_G (\Upsilon ) $ and a morphism $\Psi $ in reverse direction.
  The
morphism $\Phi $ is constructed by setting $ \Phi (S_0 ) = s_1 ;
\Phi (S_i ) = s_{i ,i+1 ;0 }\  for\  i\geq 1   ; \Phi (E_0 ) = e_1 ;
\Phi (E_i )= e_{i,i+1 ,0}
 \ for\ i\geq 1.$
It isn't hard to certify that $\Phi $ satisfying all relations in
Definition 8.3, so $\Phi $ can extend to a morphism. To define $\Psi
$, the main step is still the definition of $\Psi (e_i ) $ and $\Psi
(e_{i,j;a })$. The following lemma shows there are well defined
elements $F_i $'s and $F_{i,j;a }$'s such that if we set $\Psi (w )
= w  , \Psi (e_i ) = F_i , \Psi (e_{i,j;a}) = F_{i,j;a } ,$ Then
$\Psi $ can extend to a morphism from $B_G (\Upsilon )$ to $B^{'} _G
(\Upsilon )$ by certifying that it all relations in Definition 1.1.
The proof is almost the same as in proof of Theorem 8.5 so we skip
it.

\end{proof}

\textbf{Construction of $F_i $ and $F_{i,j;a }$} \ \ \ \ The
following lemma is similar to Lemma 8.1.

\begin{lem}
Let $G$ be the pseudo reflection group of type $G(m,1,n)$. Suppose
$G$ acts on a set $S$ and suppose there is a subset  $\{v_0 ,v_1
,\cdots , v_{n-1} \} \subset S$ such that: \begin{align*} &(1) ( S_0
)^i S_1 ( S_0 )^i  (v_1 ) =v_1  \ for\  any\  i;   &(2)& S_1 ( S_0
)^i S_1 (v_0 )
=v_0 \ for\  any\  i,\\
 &(3) S_i (v_j ) = v_j \  if\  \mid i-j \mid \geq 2 ,
&(4)& S_i S_{i+1} (v_i ) = v_{i+1 }  \ for\  i\geq 1,\\
 &(5) S_{i} S_{i-1}(v_i ) = v_{i-1 } \ for\  i\geq 2 , &(6)& S_i (v_i ) =v_i .
\end{align*}
Then the identity $w ( \mathbb{H} _i ) = \mathbb{H} _j $ implies $w
(v_i ) =v_j $, where $w\in G$, and $S_i $'s are generators of $G$ as
in Proposition 8.1. For convenience  here we denote $H_1 $ as
$\mathbb{H} _0 $, $H_{i,i+1 ;0 }$ as $\mathbb{H} _i $ for $i\geq 2$.
\end{lem}

\begin{proof}
In this case instead of using Artin monoid we prove it by direct
computation. Let $\bar{v} _i = S_{i-1} \cdots S_1 (v_0 )$ for $1\leq
i\leq n $;  $\bar{v}^a _{i,j} = (S_{j-1} \cdots S_1 S_0 S_1 \cdots
S_{j-1} )^a S_{j-1} \cdots S_{i+1} (v_i )$ for $i<j-1 $; $ \bar{v}
^a _{i,i+1} = (S_i \cdots S_0 \cdots S_i ) ^a (v_i )$. The following
identities show that the set $\{ \bar{v} _i \}_{1\leq i\leq n} \cup
\{ \bar{v} ^a _{i,j } \}_{i<j }$ is closed under the action of $G$,
and the map $J: \mathcal {A} \rightarrow \{ \bar{v} _i \}_{1\leq
i\leq n} \cup \{ \bar{v} ^a _{i,j } \}_{i<j }:$ $ H_{i,j;a}\mapsto
\bar{v}^a _{i,j};\ \ \ H_i \mapsto \bar{v} _i$ is $G$ equivariant.
Thus the lemma is proved.

(a) $S_0 (\bar {v} _i ) = S_{i-1} \cdots S_2 S_0 S_1 (v_0 ) =
S_{i-1} \cdots S_2 S_1 \cdot S_1 S_0 S_1 (v_0 ) = S_{i-1} \cdots S_2
S_1 (v_0 ) = \bar{v} _i $.

(b)$S_0 (\bar{v} ^a _{1,i} )= S_0 (S_{i-1} \cdots S_0 \cdots
S_{i-1})^a S_{i-1} \cdots S_2 (v_1 ) = (S_{i-1} \cdots S_0 \cdots
S_{i-1})^a S_{i-1} \cdots S_2 S_0 (v_1 )$

$= (S_{i-1} \cdots S_0 \cdots S_{i-1})^{a-1} S_{i-1} \cdots S_1 S_0
S_1 S_0 (v_1 ) = (S_{i-1} \cdots S_0 \cdots S_{i-1})^{a-1} S_{i-1}
\cdots S_2 (v_1 ) $

$= \bar{v} ^{a-1} _{1,i}.$

(c) If $i\geq 2$, $S_0 (\bar{v} ^a _{i,j}) = S_0 (S_{j-1} \cdots S_1
S_0 S_1 \cdots S_{j-1} )^a S_{j-1} \cdots S_{i+1} (v_i )$

$= (S_{j-1} \cdots S_1 S_0 S_1 \cdots S_{j-1} )^a S_{j-1} \cdots
S_{i+1} S_0 (v_i )= \bar{v} ^a _{i,j}.$

(d) For $i\geq 1$. $S_i (\bar{v} _i )= S_i S_{i-1} \cdots S_1 (v_0 )
= \bar{v} _{i+1} .$

(e) $S_i (\bar{v} _{i+1})= \bar{v} _{i}$  for $i\geq 1$. (Equivalent
to (d) )

(f) $S_i (\bar{v} _j ) = \bar{v} _j $ if $i\neq 0 $ and $j\neq i,i+1
.$

$j\neq i,i+1  \Leftrightarrow i>j \ or\ i<j-1 .$ If $i >j $ then
$S_i (\bar{v} _j )= S_i S_{j-1} \cdots S_1 (v_0 )$

$= S_{j-1} \cdots S_1 S_i (v_0 ) = \bar{v} _j $; If $i< j-1 $ then
$S_i (\bar{v} _j )= S_i S_{j-1} \cdots S_1 (v_0 )$

$= S_{j-1} \cdots S_{i} S_{i+1} S_i \cdots S_1 (v_0 ) =  S_{j-1}
\cdots S_{i+1} S_{i} S_{i+1} \cdots S_1 (v_0 )$

$= S_{j-1} \cdots S_1 S_{i+1} (v_0 ) =\bar{v} _j .$

(g) $S_i (\bar{v} ^a _{i,l} ) = \bar{v} ^a _{i+1 ,l} $ if $l\geq i+2
.$

First we have

$S_i (S_{l-1} \cdots S_0 \cdots S_{l-1} )= S_{l-1} \cdots S_i
S_{i+1} S_i \cdots S_0 \cdots S_{l-1}$

$ =  S_{l-1} \cdots S_{i+1} S_i S_{i+1} \cdots S_0 \cdots S_{l-1} =
S_{l-1} \cdots S_0 \cdots S_{i+1} S_{i} S_{i+1} \cdots S_{l-1} $

$= (S_{l-1} \cdots S_0 \cdots S_{l-1} )S_i .$

So

$S_i (\bar{v} ^a _{i,l} )= S_i (S_{l-1} \cdots S_0 \cdots S_{l-1} )
^a S_{l-1} \cdots S_{i+1} (v_i ) = (S_{l-1} \cdots S_0 \cdots
S_{l-1} ) ^a S_i S_{l-1}  \cdots S_{i+1} (v_i ) $

$= (S_{l-1} \cdots S_0 \cdots S_{l-1} ) ^a S_{l-1} \cdots S_{i+2}
S_i S_{i+1} (v_i ) $

$=(S_{l-1} \cdots S_0 \cdots S_{l-1} ) ^a S_{l-1} \cdots S_{i+2}
(v_{i+1} ) = \bar{v} ^a _{i+1 ,l} .$

(h) $S_i (\bar{v} ^a _{l,i}) = \bar{v} ^a _{l, i+1}$ for $l<i $.

$S_i (\bar{v} ^a _{l,i})= S_i (S_{i-1} \cdots S_0 \cdots S_{i-1} )
^a S_{i-1} \cdots S_{l+1}(v_l )$

$ = (S_i \cdots S_0 \cdots S_i ) ^a S_i S_{i-1} \cdots S_{l+1} (v_l
) = \bar{v} ^a _{l,i+1}$.

(i) $S_i (\bar{v} ^a _{k,l} ) = \bar{v} ^a _{k,l} $ if $\{k,l \}
\cap \{ i,i+1 \} = \emptyset .$

 $S_i (\bar{v} ^a _{k,l} )= S_i (S_{l-1} \cdots S_0 \cdots
S_{l-1})^a S_{l-1} \cdots S_{k+1} (v_k )$

\end{proof}

Let $\mathbb{E}^{'}= \{ w E_i w^{-1}  \}_{w\in G, 0\leq i\leq n-1}
$. $G$ acts on $\mathbb{E}^{'}$ by conjugation. This action
satisfies the conditions of Lemma 8.5 if let $\{E_0 ,E_1 ,\cdots
,E_{n-1} \}$ to be $\{v_0 ,v_1 ,\cdots ,v_{n-1} \}$ in the lemma.
 By this lemma there is an unique $G$-equivariant map $F$: $\mathcal
 {A} \rightarrow \mathbb{E}^{'}$ such that $F(H_1 )= E_0 $, $F(H_{i,i+1;0})=
 E_i$ for $1\leq i\leq n-1$. Define $F_i = F(H_i )$ for $1\leq i \leq n$, and
 $F_{i,j;a}= F(H_{i,j;a})$.

 By comparing Definition 2.1 (of the cyclotomic Brauer algebra )
 with Definition 8.3, we have the following theorem.

 \begin{thm} In the data $\Upsilon$ if $\mu =1$, $\mu_a =\sigma_a $
 $(1\leq a\leq m-1 )$ and $\sigma_0 =\tau_1$, then

 (1) Set a  map $\Phi$ by : $S_i \mapsto s_i $ $(1\leq i\leq n-1)$, $S_0 \mapsto
 t_1$,  $E_i \mapsto e_i$ $(1\leq i\leq n-1)$, $E_0 \mapsto 0$, then
 $\Phi$ extends to a surjective morphism from $B^{'} _G (\Upsilon)$
 to $\mathscr{B} _{m,n}(\delta )$ and $\ker \Phi = <E_0 >$, the idea
 generated by $E_0$.

 (2) Set a map $\Psi$ by : $s_i \mapsto S_i $ $(1\leq i\leq n-1)$,
 $t_i \mapsto S_{i-1}\cdots S_1 t_1 S_1 \cdots S_{i-1}$ $(1\leq i\leq
 n)$, $e_i \mapsto E_i $ $(1\leq i\leq n-1)$, then $\Psi$ extends to
 an morphism from  $\mathscr{B} _{m,n}(\delta )$ to $B^{'} _G
 (\Upsilon)$. We have $\Phi \circ \Psi = id$, so
 $B^{'} _G (\Upsilon)\cong \mathscr{B} _{m,n}(\delta ) \oplus <E_0 >
 $.

 \end{thm}

\section{ Conclusions }

If we take off the relation $(1)^{'}$ in Definition 1.1, we obtain
an algebra bigger than $B_G (\Upsilon)$. Denote this algebra as
$\bar{B}_G (\Upsilon)$.  It is easy to see $\bar{B}_G (\Upsilon)$
coincide with $B_G (\Upsilon)$ if $G$ is a simply laced Coxeter
group. The algebra $\bar{B}_G (\Upsilon)$ also satisfy hypothesis 1
and 2 in Section 1. We can prove $\bar{B}_G (\Upsilon)$ is finite
dimensional if $G$ is a finite group.  In general, $\bar{B}_G
(\Upsilon)$ has $B_G (\Upsilon)$ as a genuine quotient, thus contain
more irreducible representations.

We can ask the following questions. If $B_G (\Upsilon)$ are
cellular, or generically semisimple, or have invariant dimension for
any finite $G$?  Does $B_G (\Upsilon)$ has affine cellular structure
in the sense of Konig and Xi \cite{KX} when $G$ is an affine Coxeter
group? How to deform $B_G (\Upsilon)$ by using the associated KZ
connection?

In \cite{CGW1} the authors mentioned the perspective of application
of generalized BMW algebras in representation theory. Beside of it
we'd like to mention that through analysis those Brauer type
algebras we can obtain some new flat $G-$invariant connections on
the complementary spaces $M_G$, thus obtain some new representations
of the corresponding Artin group or complex braid group, just as the
case of $H_3$ type.  In the same time by solving the equations of
flat sections we would encounter with some new fuchs equations on
$M_G$ and some new hypergeometric type functions.


{\small

\hspace{-0.70cm} {\sc Department of Mathematics}

\nd{\sc University of Science and Technology of China}

\nd {\sc Hefei 230026 China}

\nd {\sc E-mail addresses}: {\sc Zhi Chen} ({\tt
zzzchen@ustc.edu.cn}).

\end{document}